\documentclass[12pt]{amsart}
\usepackage[hmargin=1in,vmargin=1in]{geometry}
\usepackage{amsmath,amsfonts,amssymb}
\usepackage{amsthm}

\usepackage{color} 
\usepackage{enumerate,graphics, graphicx}
\usepackage{enumitem}
\usepackage{tikz}
\usepackage[all,cmtip]{xy}
\usetikzlibrary{arrows}
\usepackage[pdftex,bookmarks=true]{hyperref}
\usepackage{verbatim}
\usepackage{caption}
\usepackage{subcaption} 

\newtheorem{thm}{Theorem}[section]
\newtheorem{lem}[thm]{Lemma}

\newtheorem{prop}[thm]{Proposition}
\newtheorem{cor}[thm]{Corollary}

\theoremstyle{definition}
\newtheorem{defn}[thm]{Definition}

\newtheorem{eg}[thm]{Example}

\newtheorem{rem}[thm]{Remark}

\newtheorem*{prop_lcr_diff1}{Proposition \ref{prop:LCR_diff_one}}
\newtheorem*{prop_max_deg_lcr}{Proposition \ref{prop:max_deg_lcr}}

\usepackage{circuitikz}
\usepackage{float}

\newcommand{\shape}{\mathrm{shape}}

\def\deg{\text{deg}}

\def\C{\mathbb{C}}

\def\R{\mathbb{R}}
\def\P{\mathbb{P}}

\usepackage{mathabx}

\newcommand{\bfc}{\mathbf{c}}
\newcommand{\bfd}{\mathbf{d}}

\newcommand{\bfR}{\mathbf{R}}
\newcommand{\bfC}{\mathbf{C}}
\newcommand{\bfL}{\mathbf{L}}

\title[Identifiability of Series-Parallel LCR Systems]{Structural Identifiability of Series-Parallel LCR Systems}

\author[Bortner]{Cashous Bortner}
\address{North Carolina State University}
\author[Sullivant]{Seth Sullivant}
\address{North Carolina State University}

\begin{document}

\begin{abstract}
We consider the identifiability problem for the parameters of 
series-parallel LCR circuit networks.  We prove that 
for networks with only two classes of components
 (inductor-capacitor (LC), inductor-resistor (LR), 
and capacitor-resistor (RC)), the parameters are identifiable if and only 
if the number of non-monic coefficients of the constitutive equations 
equals the number of parameters.  The notion of the ``type'' of the constitutive equations
plays a key role in the identifiability of LC, LR, and RC networks.
We also investigate the general series-parallel LCR circuits (with all three
classes of components), and classify the types of constitutive equations
that can arise, showing that there are $22$ different types.  However,
 we produce an example that shows
that the basic notion of type that works to classify identifiability
of two class networks is not sufficient to classify the identifiability
of general series-parallel LCR circuits.
\end{abstract}

\maketitle

\section{Introduction} \label{sec:intro}

LCR circuits, also referred to as LCR systems or models, are electrical circuits consisting of 
networks of inductors, capacitors, and resistors, which we call \textit{base elements}.  
These circuits have a wide array of applications, most notably in communications systems, 
such as filters and tuners used in television and radio tuning \cite{wang_2010}.  
Also of interest are the circuits generated by two base element types, 
for example simple LR circuits can be made into high-pass (or low-pass) filters which pass
high frequencies through the circuit with minimal dampening, while low frequencies are not 
able to pass as a result of strong dampening \cite{huning_felix_2014}.  

Each of the base elements in an LCR system has a defining parameter which are referred to as the inductance ($L$), 
capacitance ($C$), and resistance ($R$) respectively.  The system as a whole also has measurable state variables
called the voltage ($V$) and the current ($I$).  
A natural question emerging from the study of 
LCR systems is whether or not we can determine the parameter values of each of the base 
elements given the measurements of the voltage and current over time over the whole system, 
and in particular if we can do so uniquely.  

\textit{Structural identifiability}  is the study of which parameters of a model  can be determined uniquely
from its input and output dynamics.  If we are able to determine all of the model's 
parameters uniquely, we say that the model is \textit{globally identifiable}.  
If, however, we are only able to determine uniqueness of the parameters up to a finite number of values
(or, equivalently, in a small neighborhood of the original parameter values), 
we  say that the model is \textit{locally identifiable}.  If we cannot 
recover the parameters uniquely up to a finite number of values, the model is said to be \textit{unidentifiable}.  The study of the structural identifiability of models has been of interest since the work of Bellman and  \AA{}str\"om \cite{bellman_astrom_1970}.  Since then, several different systems ranging from physical to biological have been analyzed for identifiability using various techniques \cite{chis_banga_balsa_2011, little_heidenreich_li_2010, miao_xia_perelson_wu_2011, WALTER19811}.

In this paper, we study the structural identifiability of LCR circuits where the underlying
network of components is a \emph{series-parallel graph}.  In the cases where there is a series-parallel
LCR network that only involves two types of components (i.e.,  inductor-resistor systems, 
capacitor-resistor systems, or inductor-capacitor systems) we give a complete characterization of
when these models are identifiable.  In particular, we have the following theorem.

\begin{thm}
Let $\mathcal{N}$ be a series-parallel LCR network that involves only two types of components.
Then the network is locally identifiable if and only if the constitutive equation of the model
has as many non-monic coefficients as the there are parameters.
\end{thm}

We also give explicit combinatorial conditions on the series-parallel constructions that
guarantee local identifiability in two component type LCR models, which are summarized in
certain ``multiplication tables''.  In each case (LC, LR, or RC), there are 4 types of 
systems characterized by the shape of the constitutive equation, and identifiability of a series or
parallel combination is completely determined by which pairs of types are joined together.

We also begin the study of general series-parallel LCR circuits.  These turn out to be much more
complicated because there are LCR systems where the number of non-monic coefficients is larger than
the number of parameters.  As a result, in addition to the identifiability problem, there are also interesting
questions about the constraints on the coefficients that arise.  Our analysis of the structure of the
constitutive equations shows that for general series-parallel LCR systems, there are
$22$ different types of constitutive equations.  With that being said, it remains an open problem
to determine complete identifiability rules of general series-parallel LCR circuits.

Part of our motivation for pursuing this project comes from past work of Mahdi, Meshkat and the
second author \cite{mahdi_meshkat_sullivant_2014}, which characterized the identifiability of
series-parallel viscoelastic systems whose elements consists of springs and dashpots.  
The \textit{electromechanical analogy}, sometimes called the impedance analogy, 
shows that this is the same as studying identifiability of the RL systems.
We were interested in generalizing those results to the three component
systems whose study we begin in this paper.

The organization of this paper is as follows: Section \ref{sec:LCRoverview}
gives background on LCR systems.  Section \ref{sec:ident} defines identifiability
and introduces its study in the context of LCR systems.  Section \ref{sec:proj}
discusses the perspective of projective geometry for studying circuit models,
and uses this to prove a duality result.  Section \ref{sec:RL/RC}
describes results of the two-element systems containing only resistors and inductors, as well as the two-element systems containing only resistors and capacitors.  
Section \ref{sec:LC} presents results for the two-element systems containing only inductors and capacitors.  Section \ref{sec:LCR} describes results of the general LCR systems.  
Section \ref{sec:equations} introduces the problem of studying the equations that
define the vanishing ideal of an LCR circuit model.
Section \ref{sec:future} outlines some paths for future study.


\section{LCR Systems}\label{sec:LCRoverview}
The ideal resistor follows Ohm's law which describes a relationship between the voltage ($V$) across the resistor, and the current ($I$).  In the case of the resistor, the voltage and current are proportional with constant of proportionality $R$ which is referred to as the \textit{resistance}, which we write as:
\begin{equation}\label{eq:res}
V = R I.
\end{equation}
Similarly, the ideal inductor exhibits the following relationship between the voltage and the derivative with respect to time of the current:
\begin{equation}\label{eq:ind}
V = L \dot{I}
\end{equation}
where $L$ is called the \textit{inductance}.
The ideal capacitor is often considered the dual of the inductor, where the relationship between the time derivative of the voltage and current is described by
\begin{equation}\label{eq:cap}
\dot{V}=CI
\end{equation}
where $C$ is  the \textit{inverse capacitance}.  Note that we use $\dot{V}=CI$
instead of the more familiar $C\dot{V}=I$ for mathematical convenience. 
 For this reason, $C$ in this paper
is the inverse of the capacitance.  This change will not affect results of identifiability.

We call these equations relating the voltage and current of LCR systems \textit{constitutive equations}.  In general, we can use Kirchhoff's Current and Voltage Laws to generate a single constitutive equation of circuits consisting of parallel and series combinations of these three base elements. 

\begin{thm}[Kirchhoff's Current Law]   The algebraic sum of the currents entering any node is zero, 
i.e.~the net current flowing into and out of any node must be zero.
\end{thm}
\begin{thm}[Kirchhoff's Voltage Law]
The algebraic sum of the voltages around any loop is zero.  
\end{thm}

\begin{eg}\label{ex:series}
Consider the series combination of a resistor and an inductor. 

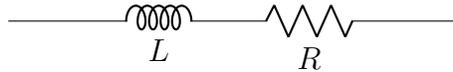
\vspace*{5mm}
\begin{figure}[H]
\begin{circuitikz}
\draw 
(1,0) to (2,0) 
 to[L,a=$L$] (4,0) to[R,a=$R$] (6,0) to (7,0)
;\end{circuitikz}
\end{figure}

By Kirchhoff's Voltage Law, we get that the voltage over the whole system $V$ must be the sum of the voltages over each element in the system, i.e.~
\[
V = V_L+V_R=L\dot{I_L}+RI_R.
\]
Also, by Kirchhoff's Current Law, we know that the net current of the system must be equal to the current of each element, i.e.~$I_L=I_R=I$.  Therefore, we get that the constitutive equation describing this circuit is
\[
V=L \dot{I}+RI.
\]
\end{eg}

\begin{eg}\label{ex:para}
Now consider a parallel combination of a resistor and an inductor.

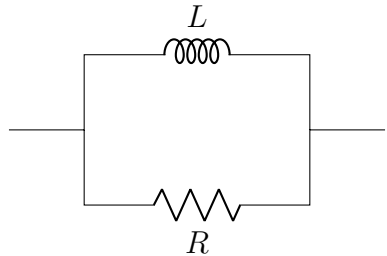
\begin{center}
\begin{circuitikz}
\draw 
(1,0) to (2,0) to (2,1)
 to[L,l=$L$] (5,1) to (5,0)   
(2,0) to (2,-1) to[R,a=$R$] (5,-1) to (5,0) to (6,0)
;\end{circuitikz}
\end{center}
\end{eg}

By Kirchhoff's voltage law, the sum of the voltage around the parallel loop must be zero, hence $V_L=V_R=V$.  Also, by Kirchhoff's current law, the current of the system is the sum of each of the currents, i.e.~$I=I_R+I_L$.

Taking the time based derivative of this current sum, along with the time based derivative of the resistor constitutive equation, we get
\[
\dot{I} = \dot{I_R}+ \dot{I_L}=\frac{1}{R} \dot{V_R} + \frac{1}{L} V_L.
\]
Thus, the constitutive equation the system is
\[
\dot{I} = \frac{1}{R} \dot{V} + \frac{1}{L} V.
\]

A natural question to ask is how we can generate these constitutive equations for more complex systems.  Suppose $S_1$ and $S_2$ represent two circuits with respective constitutive equations $f_1 V_1 = f_2 I_1$ and $f_3 V_2 = f_4 I_2$ where $f_i$ are all linear differential operators.  We can write these differential operators as 

\begin{align}
f_1 &= a_{n_1} \frac{d^{n_1}}{dt^{n_1}} + \cdots + a_{m_1} \frac{d^{m_1}}{dt^{m_1}} \nonumber \\
f_2 &= b_{n_2} \frac{d^{n_2}}{dt^{n_2}} + \cdots + b_{m_2} \frac{d^{m_2}}{dt^{m_2}} \label{eq:fourfs} \\
f_3 &= c_{n_3} \frac{d^{n_3}}{dt^{n_3}} + \cdots + c_{m_3} \frac{d^{m_3}}{dt^{m_3}} \nonumber\\
f_4 &= d_{n_4} \frac{d^{n_4}}{dt^{n_4}} + \cdots + d_{m_4} \frac{d^{m_4}}{dt^{m_4}} \nonumber
\end{align}

Now we will consider parallel and series combination of the systems $S_1$ and $S_2$, and derive the resulting constitutive equation from those of $S_1$ and $S_2$.

\begin{prop}[Series Combination]\label{prop:series_con}
The series combination of two LCR systems $S_1$ and $S_2$ with respective constitutive equations $f_1V_1=f_2I_1$ and $f_3V_2=f_4I_2$ has constitutive equation
\[
f_1f_3V=(f_1f_4+f_2f_3)I.
\]
\end{prop}

\begin{proof}
Let $T$ be the series combination of two LCR systems $S_1$ and $S_2$  with respective constitutive equations $f_1V_1=f_2I_1$ and $f_3V_2=f_4I_2$.

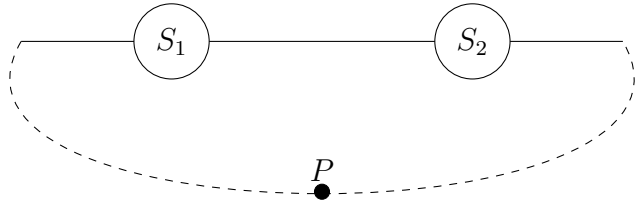
\begin{figure}[H]
\begin{tikzpicture}
\draw (0,0) circle (.5) node {$S_1$};
\draw (4,0) circle (.5) node {$S_2$};
\draw (.5,0) -- (3.5,0);
\draw (4.5,0) -- (6,0);
\draw (-.5,0) -- (-2,0);
\draw[dashed] (-2,0) to[out=-120,in=-60]  (6,0); 
\draw (2,-2) node[anchor=south] {$P$};
\filldraw (2,-2) circle (.1);
\end{tikzpicture}
\end{figure}

Note that by Kirchhoff's Current Law, the node $P$ between the two systems must have a net zero incoming current. Therefore, the current of either system must be the same and this current will also be the current of the new system $T$, i.e.~$I_1=I_2=I$.  Similarly, by Kirchhoff's Voltage Law, the voltage on the loop, which in this case is the whole system, must sum to the voltage of the system, i.e.~$V=V_1+V_2$.  If $f_1$ and $f_3$ are relatively prime, then we get
\begin{align}
V&=V_1+V_2 \nonumber \\
V&= \frac{f_2}{f_1} I_1 + \frac{f_4}{f_3}I_2 \nonumber \\
(f_1f_3) V &= (f_1f_4+f_2f_3)I \label{eq:con_series}
\end{align}
Thus, the series combination of the two systems $S_1$ and $S_2$ has constitutive equation of the form in Equation \ref{eq:con_series}.
\end{proof}

\begin{prop}[Parallel Combination]\label{prop:para_con}
The parallel combination of two LCR systems $S_1$ and $S_2$ with respective constitutive equations $f_1V_1=f_2I_1$ and $f_3V_2=f_4I_2$ has constitutive equation
\[
(f_1f_4+f_2f_3)V=f_2f_4I
\]
\end{prop}

\begin{proof}
Let $T$ be the parallel combination of two LCR systems $S_1$ and $S_2$  with respective constitutive equations $f_1V_1=f_2I_1$ and $f_3V_2=f_4I_2$.

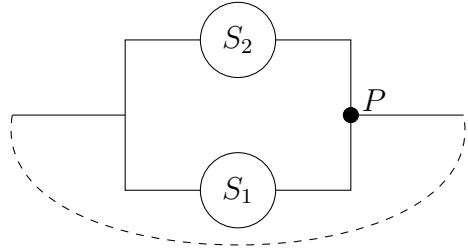
\begin{figure}[H]
\begin{tikzpicture}
\draw (2,0) circle (.5) node {$S_1$};
\draw (2,2) circle (.5) node {$S_2$};
\draw (2.5,2) -- (3.5,2) -- (3.5,0) -- (2.5,0);
\draw (1.5,2) -- (0.5,2) -- (0.5,0) -- (1.5,0);
\draw (3.5,1) -- (5,1);
\draw (0.5,1) -- (-1,1);
\draw[dashed] (-1,1) to[out=-100,in=-80]  (5,1); 
\draw (3.5,1.2) node[anchor=west] {$P$};
\filldraw (3.5,1) circle (.1);
\end{tikzpicture}
\end{figure}

Again, by Kirchhoff's Voltage Law, we get that the total voltage around the parallel combination loop must be net zero, i.e.~$V_1-V_2=0$.  Also the voltage around the entire system must be net zero, thus $V=V_1=V_2$.  Kirchhoff's Current Law states that the node $P$ must have a net zero incoming current, i.e.~$I-I_1-I_2=0$, hence $I=I_1+I_2$.  Thus, we get that the parallel combination of two systems $S_1$ and $S_2$ has constitutive equation 
\begin{equation}\label{eq:con_para}
(f_1f_4 + f_2f_3)V = (f_2f_4)I.
\end{equation}

\end{proof}

\begin{eg}\label{eg:lcr_series}
Consider the series combination of each of the three base elements of an LCR system, namely a resistor with resistance $R$, a capacitor with inverse capacitance $C$, and an inductor with inductance $L$.

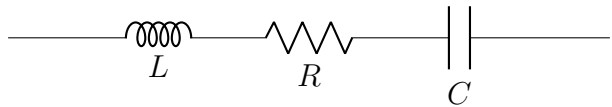
\vspace*{5mm}
\begin{figure}[H]
\begin{circuitikz}
\draw 
(1,0) to (2,0) 
 to[L,a=$L$] (4,0) to[R,a=$R$] (6,0) to[C,a=$C$] (8,0) to (9,0) 
;\end{circuitikz}
\end{figure}

The constitutive equation for this model is
\begin{equation}\label{eq:lcr_series_eg}
\dot{V} = L\ddot{I}+R \dot{I} + CI.
\end{equation}
\end{eg}

The special structure of series-parallel networks means that it is possible to use
notation purely in equations to represent a series-parallel LCR circuit, rather than
necessarily using a figure.  Specifically, we can use the notation $M \vee N$ to denote
the parallel combination of networks $M$ and $N$, and $M \wedge N$ to denote the series
combination.  The base elements can be written using the symbols for their respective parameters.
For example, the network in Example \ref{eg:lcr_series} can be represented as
\[
L \wedge  R \wedge C.
\]
Note that the $\wedge$ and $\vee$ operations are commutative and associative in terms of their 
relations for producing new networks, but they do not satisfy a distributive law.


\section{Identifiability}\label{sec:ident}

In the study of mathematical models, a common question is whether we can 
estimate the values of parameters only from the measurable input and output into the 
system, and more specifically can they be estimated uniquely.  This is the question 
of \textit{structural identifiability.}  

In the case of LCR systems, we consider the identifiability of the coefficient map 
$\mathbf{c} \colon \R^n \to \R^{m}$, mapping the parameters to the coefficients of the constitutive equation. 
Here we make the standard assumption that it is possible with perfect data to recover the constitutive equation.  For a more precise understanding of the relationship between measured data and the defining equation of a model, refer to Ovchinnikov et al.~ \cite{hong_ovchinnikov_pogudin_yap_2020, ovchinnikov_pillay_pogudin_scanlon_2021, ovchinnikov_pogudin_thompson_2021}.  Since this equation is only determined up to a constant factor, we globally assume that one of the coefficients is fixed to the value one, which makes the constitutive equation monic.
We now formally define identifiability:

\begin{defn}
Let $\mathbf{c}$ be a function $\mathbf{c} \colon \theta  \to \R^m$, 
where $\theta \subseteq \R^{n}$ is the parameter space.  The model is 
\textit{globally identifiable} from $\mathbf{c}$ if and only if the map $\mathbf{c}$ is one-to-one.  
The model is \textit{locally identifiable} from $\mathbf{c}$ if and only if the map $\mathbf{c}$ is 
finite-to-one.  The model is \textit{unidentifiable} from $\mathbf{c}$ if and only
 if $\mathbf{c}$ is infinite-to-one.
\end{defn}

\begin{rem}
Suppose that an LCR system $\mathcal{M}$ has  coefficient map $\mathbf{c} \colon \R^n \to \R^{m}$ where $n$ represents the number of parameters, and $m$ represents the number of non-monic, nontrivial coefficients in the constitutive equation.  Note then that if $m <n$, that is, there are more parameters than non-monic, nontrivial coefficients in the constitutive equation, then  $\mathcal{M}$
must be unidentifiable, since $\bfc$ is automatically infinite-to-one.
\end{rem}

\begin{eg}[Example \ref{eg:lcr_series} continued]  \label{ex:LCR}
Recall that Example \ref{eg:lcr_series} has constitutive equation of the series combination of each of the three base elements of an LCR system shown in Equation \ref{eq:lcr_series_eg}, $\dot{V} = L \ddot{I} + R \dot{I} + CI$.  In this case, we get a simple coefficient map of 
\begin{align*}
\mathbf{c} \colon \R[L,C,R] &\to \R[c_2,c_1,c_0] \\
\{R,L,C\} &\mapsto \{L,R,C\}.
\end{align*}
Thus,  each of the parameters $L,R$ and $C$ are identifiable by the coefficient map.   
\end{eg}


\section{Projective Geometry and  Circuit Duality}\label{sec:proj}

In this section, we introduce a perspective based on projective geometry.
This provides us a useful framework for discussing identifiability
that avoids the use of non-monic coefficients.  It also allows
for a straightforward duality results about the interchange of 
capacitors and inductors.

Recall that the projective space $\P\R^n$ is defined as the set of lines through
the origin in $\R^{n+1}$.  Formally, we can take the set $\R^{n+1} \setminus \{0\}$ and
consider the equivalence relation $\sim$ that $x \sim y$ if there is a 
$ \lambda \in \R\setminus \{0\}$ such that $x = \lambda y$.  Then 
$\P\R^n$   is equal to the set of equivalence classes:
\[
\P\R^n  =  (\R^{n+1} \setminus \{0\} )/ \sim  .
\]

In this paper, we consider linear differential equations.  To this end, the set
of all differential equations with a given shape is naturally considered
as a projective space.  Indeed, if $L_1 V = L_2 I$ is a differential equation
coming from a particular LCR circuit, and $\lambda$ is any nonzero constant, then
$\lambda L_1 V  =  \lambda  L_2 V$ describes the same dynamics.
In particular, it is only possible to recover the underlying constitutive 
equation up to a constant.  The typical way that this is dealt with is to talk
about non-monic coefficients in the constitutive equation--  essentially, picking
one term to be the leading term and dividing through so the coefficient of that
term is equal to one.  This is a satisfactory approach in most situations.
We find the perspective from projective geometry can also be useful.

To start with, we consider the constitutive equations of the three basic elements:
\[
V = RI, \quad \quad  V = L \dot{I}, \quad \quad  \dot{V} = C I.
\] 
Thinking about these projectively, we would have the basic constitutive equations:
\begin{equation} \label{eq:projectiveparameters}
R_0 V = R_1 I, \quad \quad  L_0 V = L_1 \dot{I}, \quad \quad  C_0\dot{V} = C_1 I.
\end{equation}
So in projective geometry language, our parameter space for an LCR model, goes from an
$\R^k$ (in the case that there are $k$ basic elements),
to a $(\P \R^1)^k$.

\begin{eg}
Consider the LCR circuit system from Example \ref{ex:LCR}, which has three
components.  Using the projective version of the parameters from (\ref{eq:projectiveparameters})
we get the constitutive equation
\[
R_0 L_0 C_0 \dot{V} = R_0L_1C_0 \ddot{I} + R_1L_0C_0 \dot{I} +  R_0 L_0 C_1 I.
\]
This shows that the coefficient map is a map from $(\mathbb{P}^1)^3$ into $\mathbb{P}^3$, defined by
\[
([ R_0 : R_1],  [ L_0 : L_1], [C_0: C_1]) \mapsto  
( R_0 L_0 C_0: R_0 L_1 C_0: R_1 L_0 C_0 : R_0 L_0 C_1 ).
\]
We arrive at the usual constitutive equation by dehomogenizing 
this one:  specifically by the substitution
\[
R_0 = 1, \, \, R_1 = R, \, \,  L_0 = 1, \, \, L_1 = L, \, \, C_0 = 1, \, \, C_1 = C.
\]
\end{eg}

One useful application of the projective perspective is that 
it makes it possible to derive a duality result for 
identifiability of LCR systems.  The idea of duality of these systems and those like it date back to the work of Alexander Russell in 1904 with inspiration from reciprocals found in geometry, and the goal of finding ``convenient methods of making measurements or even suggest novel instruments or machines of value in electro-technics'' \cite{russell_1904}.

\begin{defn}
Let $M$ be a series-parallel LCR circuit model, expressed as a formula in terms of resistors $R_1, R_2, \ldots$, 
capacitors $C_1, C_2, \ldots$, and inductors $L_1, L_2, \ldots$, using the series and parallel operations $\wedge$ and
$\vee$.  Define the \emph{dual system} $\overline{M}$, to be expressed as a formula in terms of 
$\overline{R}_1, \overline{R}_2, \ldots $,  $\overline{C}_1, \overline{C}_2, \ldots$, and $\overline{L}_1, \overline{L}_2, \ldots$ by the following
rules:
\begin{enumerate}
\item  Swap each $\wedge$ with a $\vee$ and vice versa
\item  Each $R_i$ is replaced with a $\overline{R}_i$
\item  Each $C_i$ is replaced with a $\overline{L}_i$, and
\item  Each $L_i$ is replaced with a $\overline{C}_i$.
\end{enumerate}
\end{defn}

\begin{eg}
Consider the series-parallel network model $M =  (R_1 \wedge C_1)  \vee ( R_2 \wedge L_1)$.
The dual network is $\overline{M} =  
(\overline{R}_1 \vee \overline{L}_1)  \wedge ( \overline{R}_2 \vee \overline{C}_1)$.
The example is illustrated in Figure \ref{fig:dualnetwork}.
\end{eg}

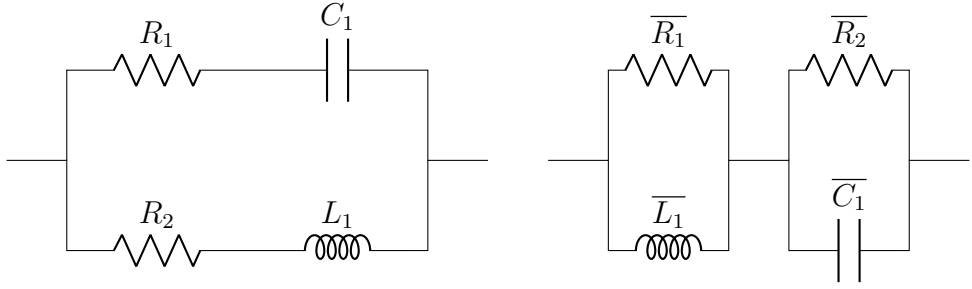
\begin{figure}
\begin{circuitikz}[scale=.8]
\draw
(0,0) to (1,0) to  (1,1.5)  to[R,l=$R_1$] (4,1.5) to[C,l=$C_1$] (7,1.5) to (7,0) to (8,0);
\draw
(1,0) to (1,-1.5) to[R,l=$R_2$]  (4,-1.5) to[L,l=$L_1$] (7,-1.5) to (7,0);
\draw
(9,0) to (10,0) to (10,1.5) to[R,l=$\overline{R_1}$] (12,1.5) to (12,0) to (13,0) to (13,1.5)
to[R,l=$\overline{R_2}$]  (15,1.5) to (15,0) to (16,0);
\draw (10,0) to (10,-1.5)  to[L,l=$\overline{L_1}$] (12,-1.5) to (12,0);
\draw (13,0) to (13,-1.5) to[C,l=$\overline{C_1}$] (15, -1.5) to (15,0);
\end{circuitikz}
\caption{A series-parallel LCR network and its dual network.}
\label{fig:dualnetwork}
\end{figure}

There is no formal difference between components of the original system $M$ and the dual system $\overline{M}$, e.g.~
a resistor $R_1$ and $\overline{R}_1$ are the same from a modeling standpoint.  However,
when we want to talk about the identifiability of these systems, it is useful to distinguish between the 
components of the original system and that of the dual system.

\begin{thm}\label{thm:duality}
Suppose that $M$ is a series-parallel LCR system and let 
$\overline{M}$ be the dual LCR system.
Then $M$ is (generically, locally) identifiable if and only if $\overline{M}$ is.
\end{thm}

To prove this, we make direct use of the projective representation of the network.
To each basic component, denoted $R_i, L_i, C_i$, we associated a projective
constitutive equation
\[
R_{0,i} V = R_i I,  \quad  L_{0,i} V = L_i \dot{I}, \quad  C_{0,i} \dot{V} = C_i I.
\]
Then on the projective representation, the duality has the effect of swapping $V$ and $I$
and $L$ and $C$.  So the dual basic constitutive equation in the projective representation
becomes
\[
 R_i  V = R_{0,i} I,  \quad  C_i  V = C_{0,i} \dot{I}, \quad   L_i \dot{V} =  L_{0,i} I.
\]
Note that affinely this corresponds to $\overline{R}_i = 1/R_i$, $\overline{L}_i = 1/C_i$
and $\overline{C}_i = 1/L_i$.

\begin{prop}\label{prop:projectivedualcons} 
Suppose that $M$ is a series-parallel LCR system with corresponding projective parameters $\mathbf{R}=(R_1,\ldots, R_r, R_{0,1},\ldots, R_{0,r} )$,  $\mathbf{L}=(L_1, \ldots , L_s, L_{0,1}, \ldots , L_{0,s})$, and $\mathbf{C}=(C_1, \ldots, C_t, C_{0,1}, \ldots, C_{0,t})$. Let $\overline{M}$ be the dual LCR system with corresponding dual projective parameters $\overline{\mathbf{R}}=( R_{0,1},\ldots, R_{0,r}, R_1,\ldots, R_r )$,  $\overline{\mathbf{L}}=(L_{0,1}, \ldots , L_{0,s}, L_1, \ldots , L_s )$,  and $\overline{\mathbf{C}} =(C_{0,1}, \ldots, C_{0,t}, C_1, \ldots, C_t).$  Let
\[
f_1(\bfR, \bfC, \bfL, \tfrac{d}{dt}) V  =  f_2(\bfR, \bfC, \bfL, \tfrac{d}{dt}) I
\]
be the constitutive equation of $M$.  Then
\[
f_2(\overline{\bfR}, \overline{\bfL}, \overline{\bfC}, \tfrac{d}{dt}) V  =  
f_1(\overline{\bfR}, \overline{\bfL}, \overline{\bfC}, \tfrac{d}{dt}) I
\]
is the constitutive equation of $\overline{M}$.
\end{prop}

\begin{proof}
The proof is by induction on the number of components.  The statement is
clearly true if there is only one component by the definition of the
duality operations.  

Suppose that $M$ has more than one component.  That means it can be broken up
as either a series or parallel combination of two other components.
We handle the case of a series combination, the case of a parallel combination being
analogous.  So suppose that $M = M_1 \wedge M_2$.
The corresponding dual LCR system is $\overline{M} = \overline{M}_1 \vee \overline{M}_2$.
Let 
\[
f_1(\bfR, \bfC, \bfL, \tfrac{d}{dt}) V_1  =  f_2(\bfR, \bfC, \bfL, \tfrac{d}{dt}) I_1
\]
\[
f_3(\bfR, \bfC, \bfL, \tfrac{d}{dt}) V_2  =  f_4(\bfR, \bfC, \bfL, \tfrac{d}{dt}) I_2
\]
be the constitutive equations of $M_1$ and $M_2$ respectively.  Thus the constitutive
equation of $M$ is
\[
(f_1f_3) (\bfR, \bfC, \bfL, \tfrac{d}{dt}) V  =  
(f_1f_4 + f_2f_3)(\bfR, \bfC, \bfL, \tfrac{d}{dt}) I
\]
By induction, the constitutive equations of $\overline{M}_1$ and $\overline{M_2}$ are
\[
f_2(\overline{\bfR}, \overline{\bfL}, \overline{\bfC}, \tfrac{d}{dt}) V_1  =  
f_1(\overline{\bfR}, \overline{\bfL}, \overline{\bfC}, \tfrac{d}{dt}) I_1
\]
\[
f_4(\overline{\bfR}, \overline{\bfL}, \overline{\bfC}, \tfrac{d}{dt}) V_2  =  
f_3(\overline{\bfR}, \overline{\bfL}, \overline{\bfC}, \tfrac{d}{dt}) I_2
\]
Since $\overline{M}$ is a parallel combination of $\overline{M}_1$ and $\overline{M}_2$
its constitutive equation is
\[
(f_1f_4 + f_2f_3)(\overline{\bfR}, \overline{\bfL}, \overline{\bfC}, \tfrac{d}{dt}) V =
(f_1f_3)(\overline{\bfR}, \overline{\bfL}, \overline{\bfC}, \tfrac{d}{dt}) I.
\]
This is clearly the desired correct form.  This proves the result for series combinations,
and the proof for a parallel combination is similar.
\end{proof}

\begin{proof}[Proof of Theorem \ref{thm:duality}]
By Proposition \ref{prop:projectivedualcons}
the coefficient map for $M$ and $\overline{M}$ is the same
except for relabeling parameters and swapping the order of some of the coefficients.
The coefficient maps clearly have the same behavior in both cases in 
terms of being one-to-one, generically one-to-one, finite-to-one, etc.
\end{proof}


\section{RL/RC System Analysis}\label{sec:RL/RC}

In this section, we consider the identifiability of series-parallel circuits
consisting of only two types of base elements: either resistor-inductor (RL) networks or resistor-capacitor (RC)
networks.  The electromechanical analogy establishes a bijection between identifiability problems
for RL-networks and identifiability problems for viscoelastic mechanical systems consisting
of springs and dashpots.  The results of \cite{mahdi_meshkat_sullivant_2014} will
be used to deduce the main identifiability result for RL series-parallel networks.
Then we use Theorem \ref{thm:duality} to deduce the analogous identifiability result for
 RC series-parallel networks.

First, consider the case of the two-element system generated by 
parallel and series combinations of inductors and resistors.  The electromechanical analogy, specifically the Maxwell or impedance analogy, yields that a system comprised of series and parallel combinations of resistors and inductors is analogous to a mechanical system consisting of series and parallel combinations of springs and dashpots \cite{stephens_bate_1966}.  The spring-dashpot system is commonly referred to as the \textit{viscoelastic model}, and has many applications, including modeling various biological systems.  The problem of identifiability of the spring-dashpot system is well studied, with the problem of determining local identifiability reduced down to counting the number of elements in the system, i.e.~parameters, and comparing that to the number of coefficients \cite{mahdi_meshkat_sullivant_2014}. 

Recall that in determining identifiability, an easy way to determine that a model is unidentifiable by the constitutive equation is to see that there are fewer coefficients than parameters, meaning a necessary condition for identifiability is that there are at least as many coefficients as parameters.  In the case of the viscoelastic system, it was shown in \cite{mahdi_meshkat_sullivant_2014} that the number of coefficients is bounded above by the number of parameters, thus the previous necessary condition for identifiability becomes that there must be exactly the same number of parameters as coefficients.  It is  then shown that this equality of the number of coefficients and parameters is in fact a sufficient condition for local identifiability via the following theorem.

\begin{thm}[Theorem 2, \cite{mahdi_meshkat_sullivant_2014}]
A viscoelastic model represented by a spring-dashpot network is locally identifiable if and only if the number of non-monic, nontrivial coefficients of the corresponding constitutive equation equals the total number of its parameters.
\end{thm}

Due to the electromechanical analogy, we can deduce the following equivalent statement in terms of RL systems:

\begin{cor}\label{cor:RL_count_ident}
An RL system is locally identifiable if and only if the number of non-monic, nontrivial coefficients of the corresponding constitutive equation equals the total number of its parameters.
\end{cor}

Via the duality of Theorem \ref{thm:duality}, we also get the following corollary.

\begin{cor}
An RC system is locally identifiable if and only if the number 
of non-monic, nontrivial coefficients of the corresponding 
constitutive equation equals the total number of its parameters.
\end{cor}

\begin{proof}
The duality operation turns an RL system into an RC system and vice versa.  Theorem
\ref{thm:duality}, shows that the RL system is identifiable if and only if the dual
RC system is identifiable.  Since the duality preserves the number of coefficients,
this follows from Corollary \ref{cor:RL_count_ident}.
\end{proof}

In general, the problem of identifiability of a model is much more difficult to answer than it is for the RC and RL systems.  We will see in Section \ref{sec:LCR} that in the case of LCR systems, we no longer have a bound on the coefficients by the number of parameters, making finding identifiability criterion considerably more difficult.

In addition to these results on identifiability and relation to the number of coefficients
in the RC/RL models, it is also possible to import from \cite{mahdi_meshkat_sullivant_2014}
precise rules for identifiability of series and parallel combinations of identifiable
models.
These are encapsulated in the identifiability multiplications for the types of
combinations of constitutive equations of different shapes. 
We do not reproduce the identifiability multiplication tables from  \cite{mahdi_meshkat_sullivant_2014} here, but we will see analogous results
for LC systems in the next section.


\section{LC System Analysis}\label{sec:LC}

Now we consider the two-element systems which contain parallel and series combinations of inductors and capacitors, i.e.~LC systems.  
To analyze the identifiability of these LC systems we first will classify these 
systems into four \textit{types} dependent upon the structure of their constitutive equations.  
Since the LC systems are specific cases of LCR systems, we can state several 
general propositions about the structure of their constitutive equations, 
which we prove in the next section.  First, we recognize an upper bound on the number of coefficients on either side of the constitutive equation of an LCR, and thus an LC system.

\begin{prop_max_deg_lcr}
\textit{The maximum order of either side of the constitutive equation of an LCR system is bounded above by the number of parameters, i.e.~base elements, in the model.
}
\end{prop_max_deg_lcr} 

Note that the previous proposition yields that the maximum number of non-monic, nonzero coefficients in the constitutive equation of an LCR system is $2n+1$, where $n$ is the number of parameters.  We can also make a statement relating the lowest and highest orders of the left-hand and right-hand sides of the constitutive equation of an LCR system.

\begin{prop_lcr_diff1}\textit{
In an LCR system, the largest orders on either side of the constitutive equation must be within one of each other.  Similarly, the smallest orders on either side of the constitutive equation must be within one of each other.}
\end{prop_lcr_diff1}

In the case of LC systems, we can actually make a slightly stronger statement.

\begin{cor} \label{cor:diff_eq_1}
In an LC system, the absolute difference of the largest order of either side of the constitutive equation is \textbf{exactly} one.  Similarly, the absolute difference of the smallest order of either side of the constitutive equation of an LC system is \textbf{exactly} one.
\end{cor}

\begin{proof} 
This is true by the exact same argument in the proof of Proposition \ref{prop:LCR_diff_one}, where the base cases are only the single inductor and single capacitor systems, and replacing any ``less than or equal to'' statements  with ``equal to'' statements.
\end{proof}

Now we make a statement about how many of the coefficients on either side of a constitutive equation of an LC system must be zero.  We introduce the idea of a constitutive equation \textit{alternating}, that is, every coefficient of even or odd order in the equation is zero.

\begin{defn}
We say that a polynomial \textit{alternates} if all odd degree or all even degree coefficients are zero.  We say a polynomial is \textit{saturated} if every coefficient
between the smallest and largest degree is nonzero.
\end{defn}

\begin{rem}
Note the difference between describing a polynomial as ``not alternating'' and ``saturated.''  
In the case of a polynomial not alternating, we could possibly still 
have coefficients of zero between the largest and smallest degree, 
we just do not have that every other coefficient is zero.
\end{rem}

Note that the product of two polynomials, both of which have this alternation property, also must alternate.  With this in mind, if both sides of two LC systems' constitutive 
equations alternate, we must have that one side of their series or parallel combination also alternates, namely the side with a single product of two previous differential operators by Propositions \ref{prop:series_con} and \ref{prop:para_con}.  
With that being said, it is not immediately clear that the side which consists of a sum of two products of the previous differential operators also alternates.  This is because although each of the products in the sum must alternate, it is possible that the powers in either alternating product have different parity, so when summed together the result does not alternate.  
In the case of LC systems, we show that this parity mismatch cannot occur.

\begin{prop}\label{prop:LC_alter}
An LC system must have both sides of its constitutive equation alternate.
\end{prop}

\begin{proof}
We proceed by induction.  Note that by our definition of alternating, the base elements inductor and capacitor are inherently alternating, since one side of either constitutive equation has a single odd power, and the other has a single even power in either case.

Suppose two LC systems $N_1$ and $N_2$ have the alternating property on either side of their constitutive equations $f_1V_1=f_2I_1$ and $f_3V_2=f_4I_2$ respectively.   
From Corollary \ref{cor:diff_eq_1}, 
we know that $f_1$ and $f_2$ have difference of highest order of one, 
hence have different parity, and similarly $f_3$ and $f_4$ have different parity.  
Note that because of the remark before the statement of this proposition, 
to show that both sides of the constitutive equation of a combination of two 
LC systems alternate, we need only show that  $f_1f_4$ and $f_2f_3$ do not have different
parity.    However, we know that $f_1$ and $f_2$ have different parity and $f_3$ and $f_4$ have
different parity.  Then $f_1f_2f_3f_4$ has to have even parity, so $f_1f_4$ and $f_2f_3$
have to have the same parity.
Thus, by induction, both sides of the constitutive equation corresponding to an LC system, must alternate.
\end{proof}

Note now that we can place an upper bound on the number of nonzero coefficients in an LC system, similar to the bound in the RC and RL systems.

\begin{thm}\label{thm:LC_coeff_bound}
The number of non-monic, nontrivial coefficients of an LC constitutive 
equation is bounded above by the number of base elements.
\end{thm}

\begin{proof}
First, note that by Proposition \ref{prop:max_deg_lcr}, the maximum order of either side of the constitutive equation of an LC system with $n$ base elements is $n$.  Also, by Corollary \ref{cor:diff_eq_1}, the maximum order of the other side of constitutive equation of an LC system is $n-1$.  By Proposition \ref{prop:LC_alter}, we know that every other coefficient on either side of the constitutive equation of an LC system must be zero, i.e.~if the maximum order on one side is $n$, then at most $\lceil \frac{n+1}{2} \rceil$ coefficients must be nonzero.  Thus, if both sides of a constitutive equation have their maximal orders $n$ and $n-1$, then the total number of nonzero coefficients is bounded above by 
\[\left\lceil \frac{n+1}{2} \right\rceil + \left\lceil \frac{n}{2} \right\rceil = n+1. \]

Thus, after normalizing, there are at most $n$ non-monic, nontrivial coefficients in the constitutive equation of an LC system.
\end{proof}

\begin{rem}
Note that to recover all $n$ parameters from an LC system with $n$ base elements, 
we need the constitutive equation defining the system to have at least 
$n$ nontrivial coefficients.  This, coupled with Theorem 
\ref{thm:LC_coeff_bound} implies that, as in the case of RL and RC systems, 
a necessary condition for identifiability of an LC system with $n$ parameters 
is that the constitutive equation has $n$ non-monic, nontrivial coefficients.  
We spend the rest of this section showing that, in fact, this is also a sufficient condition.
\end{rem}

Now we can classify identifiable LC systems into four different ``types'' depending on the difference in the largest orders and smallest orders of the left and right-hand sides of their constitutive equations.  We will define the type of the LC system with constitutive equation $f_1V=f_2I$ where 
\begin{align*}
f_1 &= a_{n_1}d^{n_1}/dt^{n_1} + \cdots +a_{m_1}d^{m_1}/dt^{m_1} \\
f_2 &= b_{n_2}d^{n_2}/dt^{n_2} + \cdots +b_{m_2}d^{m_2}/dt^{m_2} \\
\end{align*}
by the ordered pair $(m_1-m_2,\ n_1-n_2)$.  Note that by Corollary \ref{cor:diff_eq_1}, we know that there are only four possible such pairs, which we define as the following types:
\[
A := (-1,-1), \quad B:= (-1,1), \quad C:= (1,-1), \quad D:=(1,1).
\]

We now consider how to build identifiable LC systems from identifiable LC systems.  
We do this by considering the \textit{shape} of each of the 
differential operators of an identifiable LC system which we define as the ordered pair 
$[a,b]$ representing the smallest and largest order respectively of the differential operator.  
Note that depending on the parity of the number of parameters $n$ of an LC system, 
certain types cannot be identifiable.  For example, consider an LC system of type $A$, 
then for the constitutive equation to have enough coefficients to potentially be identifiable, the shape in $V$ must be 
$[0,n-1]$ and the shape in $I$ must be $[1,n]$, 
so we know that $n$ must be odd by Proposition \ref{prop:LC_alter}.  
Similarly, LC systems of type $D$ must have an odd number of parameters to potentially be identifiable, 
while LC systems of types $B$ and $C$ must have an even number of parameters to potentially be identifiable.

The following two tables give the identifiability results of the 
series and parallel combinations of all of the identifiable LC system types, 
with a count of the number of non-monic, nontrivial coefficients, as well as the 
resulting type.  Note that in the column ``Identifiable?", if there is a ``no" we already
can see that this model is unidentifiable, as there are not enough coefficients as 
compared to the number of parameters.  On the other hand, we still need
to prove that the ``yes'' entries are actually identifiable.
Proving that this is the case will occupy the rest of the section and
complete the proof of Theorem \ref{thm:LCsystem}, which is the main
result of this section.

\begin{table}[H]
\begin{tabular}{|l|l|l|l|l|l|}
\hline
Type & Shape in $V$ & Shape in $I$ & Non-monic coefficients & Identifiable? & Type \\
\hline
$(A,A)$ & $[0,n_1+n_2-2]$ & $[1,n_1+n_2-1]$ & $n_1+n_2-1$ & No & $A$ \\
\hline
$(A,B)$ & $[0,n_1+n_2-1]$ & $[1,n_1+n_2]$ & $n_1+n_2$  & Yes & $A$ \\
\hline
$(A,C)$ & $[1,n_1+n_2-2]$ & $[0,n_1+n_2-1]$ & $n_1+n_2-1$  & No & $C$ \\
\hline
$(A,D)$ & $[1,n_1+n_2-1]$ & $[0,n_1+n_2]$ & $n_1+n_2$  & Yes & $C$ \\
\hline
$(B,B)$ & $[0,n_1+n_2]$ & $[1,n_1+n_2-1]$ & $n_1+n_2$  & Yes & $B$ \\
\hline
$(B,C)$ & $[1,n_1+n_2-1]$ & $[0,n_1+n_2]$ & $n_1+n_2$  & Yes & $C$ \\
\hline
$(B,D)$ & $[1,n_1+n_2]$ & $[0,n_1+n_2-1]$ & $n_1+n_2$  & Yes & $D$ \\
\hline
$(C,C)$ & $[2,n_1+n_2-2]$ & $[1,n_1+n_2-1]$ & $n_1+n_2-2$  & No & $C$ \\
\hline
$(C,D)$ & $[2,n_1+n_2-1]$ & $[1,n_1+n_2]$ & $n_1+n_2-1$ &  No & $C$ \\
\hline
$(D,D)$ & $[2,n_1+n_2]$ & $[1,n_1+n_2-1]$ & $n_1+n_2-1$  & No & $D$ \\
\hline
\end{tabular}
\caption{All identifiable series combinations of the four types of LC systems, with resulting shapes, number of coefficients, identifiability, and type.}
\label{table:1}
\end{table}


\begin{table}[H]
\begin{tabular}{|l|l|l|l|l|l|}
\hline
Type & Shape in $V$ & Shape in $I$ & Non-monic coefficients & Identifiable? & Type \\
\hline
$(A,A)$ & $[0,n_1+n_2-1]$ & $[2,n_1+n_2]$ & $n_1+n_2-1$ & No & $A$ \\
\hline
$(A,B)$ & $[1,n_1+n_2]$ & $[2,n_1+n_2-1]$ & $n_1+n_2-1$  & No & $B$ \\
\hline
$(A,C)$ & $[0,n_1+n_2-1]$ & $[1,n_1+n_2]$ & $n_1+n_2$  & Yes & $A$ \\
\hline
$(A,D)$ & $[0,n_1+n_2]$ & $[1,n_1+n_2-1]$ & $n_1+n_2$  & Yes & $B$ \\
\hline
$(B,B)$ & $[1,n_1+n_2-1]$ & $[2,n_1+n_2-2]$ & $n_1+n_2-2$  & No & $B$ \\
\hline
$(B,C)$ & $[0,n_1+n_2]$ & $[1,n_1+n_2-1]$ & $n_1+n_2$  & Yes & $B$ \\
\hline
$(B,D)$ & $[0,n_1+n_2-1]$ & $[1,n_1+n_2-2]$ & $n_1+n_2-1$  & No & $B$ \\
\hline
$(C,C)$ & $[1,n_1+n_2-1]$ & $[0,n_1+n_2]$ & $n_1+n_2$  & Yes & $C$ \\
\hline
$(C,D)$ & $[1,n_1+n_2]$ & $[0,n_1+n_2-1]$ & $n_1+n_2$ &  Yes & $D$ \\
\hline
$(D,D)$ & $[1,n_1+n_2-1]$ & $[0,n_1+n_2-2]$ & $n_1+n_2-2$  & No & $D$ \\
\hline
\end{tabular}
\caption{All identifiable parallel combinations of the four types of LC systems, with resulting shapes, number of coefficients, identifiability, and type.}
\label{table:2}
\end{table}

\begin{rem}
Checking for identifiability of a parallel or series combination of LC systems can be done in polynomial time via Tables \ref{table:1} and \ref{table:2}.  Similarly, checking for identifiability of a series or parallel combination of RL, and thus RC, systems can be done in polynomial time via tables found in \cite{mahdi_meshkat_sullivant_2014}.
\end{rem}

\subsection*{The Alternating Shape Factorization Problem}

We now define the alternating shape factorization problem, which is analogous to the shape factorization problem as defined in \cite{mahdi_meshkat_sullivant_2014}, though this time for alternating polynomials.

\begin{defn}
The \textit{alternating shape factorization problem} for a quadruple of shapes 
\[
Q=([m_1,n_1],[m_2,n_2],[m_3,n_3],[m_4,n_4])
\] 
is defined as follows:  for a generic pair of alternating polynomials 
$(f,g)$ with $f$ monic such that $\shape(f) = [m_1+m_3,n_1+n_3]$ and 
$\shape(g) = [\min(m_1+m_4,m_2+m_3), \max(n_1+n_4,n_2+n_3)]$, 
do there exist finitely many quadruples of alternating polynomials 
$(f_1,f_2,f_3,f_4)$ with shape $f_i=[m_i,n_i]$, $f_1,f_3$ monic, and such that $f=f_1f_3$ and $g=f_1f_4+f_2f_3$?  A quadruple of shapes $Q$ is said to be \textit{alternating good} if the 
alternating shape factorization problem for that quadruple has a positive solution.
\end{defn}

\begin{prop}\label{prop:goodquad_eq_ident}
Let $M$ be the series combination of two LC systems $N_1$ and $N_2$ with respective constitutive equation $f_1V_1=f_2I_1$ and $f_3V_2=f_4I_2$ and let $f_i$ have shape $[m_i,n_i]$.  Then the LC system $M$ is locally identifiable if and only if 
\begin{itemize}
\item[(i)] $N_1$ and $N_2$ are locally identifiable, and
\item[(ii)] $([m_1,n_1],[m_2,n_2],[m_3,n_3],[m_4,n_4])$ is an alternating good quadruple.
\end{itemize}
\end{prop}

We now work toward necessary and sufficient conditions for the series combination of two LCR models to yield a good alternating quadruple, inspired by the work done following Proposition 10 in \cite{mahdi_meshkat_sullivant_2014} for the viscoelastic case.

Let $h$ and $g$ be two alternating polynomials, and note that for fixed shapes $[m_1,n_1]$ and $[m_3,n_3]$, there are at most finitely many factorization $h=f_1f_3$, with alternating $f_1$ and $f_3$ having shapes $[m_1,n_1]$ and $[m_3,n_3]$ respectively.  Thus, in fixing one of these finitely many choices of $f_1$ and $f_3$, the equation $g=f_1f_4+f_3f_2$ is a linear system in the unknown coefficients of alternating $f_2$ and $f_4$.  

For a particular polynomial $f=j_nx^n + \cdots +  j_mx^m$ with shape $[m,n]$, we can denote the coefficients of $f$ in an $n-m+1$ dimensional vector as 
\[
[f]:= \begin{pmatrix}
j_n \\ \vdots \\ j_m
\end{pmatrix}.
\]

Again, if the $f_i$ have respective shape $[m_i,n_i]$, then the vector of coefficients of $f_1f_4$ and $f_2f_3$ can be written as the following matrix products:
\[
[f_1f_4] = \begin{pmatrix}
a_{n_1} & 0 & \cdots & 0 \\
\vdots & a_{n_1} & \cdots & 0 \\
a_{m_1} & \vdots & \cdots & \vdots \\
0 & a_{m_1} & \cdots & 0 \\
\vdots & 0 & \cdots & a_{n_1} \\
\vdots & \vdots & \cdots & \vdots \\
0 & 0 & \cdots & a_{m_1} 
\end{pmatrix}
\begin{pmatrix}
d_{n_4} \\ \vdots \\ d_{m_4}
\end{pmatrix} , \quad 
[f_3f_2] = \begin{pmatrix}
c_{n_3} & 0 & \cdots & 0 \\
\vdots & c_{n_3} & \cdots & 0 \\
c_{m_3} & \vdots & \cdots & \vdots \\
0 & c_{m_3} & \cdots & 0 \\
\vdots & 0 & \cdots & c_{n_3} \\
\vdots & \vdots & \cdots & \vdots \\
0 & 0 & \cdots & c_{m_3} 
\end{pmatrix}
\begin{pmatrix}
b_{n_2} \\ \vdots \\ b_{m_2}
\end{pmatrix}.
\]

We refer to the matrix containing the coefficients of $f_1$ as $G''$ and the matrix containing the coefficients of $f_3$ as $H''$, hence the above matrix products can be represented by $G''[f_4]$ and $H''[f_2]$ respectively.  Note that this matrix $G''$ has dimension $n_1+n_4-m_1-m_4+1$ by $n_4-m_4+1$, while $H''$ has dimension $n_2+n_3-m_2-m_3+1$ by $n_2-m_2+1$.

We can nearly represent the coefficients of $g$ by adding these two products, however there could be a difference in the dimension of the largest and smallest orders of $f_1f_4$ and $f_2f_3$.  Note however that this difference is well understood, as by Corollary \ref{cor:diff_eq_1}, the difference in the largest and smallest orders of $f_1$ and $f_2$ must be at exactly one, and likewise for $f_3$ and $f_4$.  Thus, either the largest order of $f_1f_4$ is the same as the largest order of $f_2f_3$, or it is exactly two larger or smaller.  The same is also true for the smallest orders of $f_1f_4$ and $f_2f_3$. 

Thus, we will let $G'$ and $H'$ represent the matrices $G''$ and $H''$ where either has an additional two rows of zeros added to the top or bottom of their respective matrix, if necessary.  Therefore, we can now represent the coefficients of $g$ as
\[
[g]=[f_1f_4+f_2f_3] = G'[f_4]+H'[f_2] = (G' \ H') \begin{pmatrix}
[f_4] \\ [f_2]
\end{pmatrix} 
\]

Note then that this matrix $(G' \ H')$ has dimension 
\[
\max\{n_1+n_4,n_2+n_3\} - \min\{m_1+m_4,m_2+m_3\}+1 \text{ by } n_2-m_2+n_4-m_4+2.
\]

Now, note that since both $f_2$ and $f_4$ alternate, many of the entries of $\begin{pmatrix}
[f_4] \\ [f_2]
\end{pmatrix}$ are zero.  
In fact, every other entry of $[f_2]$ and $[f_4]$ are zero, hence we can eliminate both these $(n_2-m_2+n_4-m_4)/2$ rows in the vector and the corresponding columns in the matrix $(G' \ H')$, yielding the same information.  The resulting matrix which we now call $(\overline{G} \ \overline{H})$ has dimension 
\[
\max\{n_1+n_4,n_2+n_3\} - \min\{m_1+m_4,m_2+m_3\}+1 \text{ by } \frac{n_2-m_2+n_4-m_4}{2}+2.
\]

Note that every other row of the $(\overline{G} \ \overline{H})$ matrix will consist of only zeros, since the alternation property of the polynomials $f_1$ and $f_3$ yield every other diagonal of $(G' \ H')$ consists of only zeros.  Thus, we can eliminate $(\max\{n_1+n_4,n_2+n_3\} - \min\{m_1+m_4,m_2+m_3\})/2$ rows of $(\overline{G}\ \overline{H})$ and retain the same information.  We define this final reduced matrix to be $(G \ H)$, and note that it has dimension:
\begin{equation}\label{eq:dim_GH}
\frac{\max\{n_1+n_4,n_2+n_3\} - \min\{m_1+m_4,m_2+m_3\}}{2}+1 \text{ by } \frac{n_2-m_2+n_4-m_4}{2}+2.
\end{equation}


We now determine when the alternating shape factorization problem has finitely many solutions.

\begin{prop} \label{prop:goodquad_if_invert}
The quadruple $([m_1,n_1],[m_2,n_2],[m_3,n_3],[m_4,n_4])$ for the four alternating polynomials is alternating good if and only the matrix $(G \ H)$ is invertible.
\end{prop}

\begin{proof}
We can write the shape factorization problem of $([m_1,n_1],[m_2,n_2],[m_3,n_3],[m_4,n_4])$ in the matrix factored form $G'[f_4]+H'[f_2]=[g]$, where every other coefficient will be zero.  Thus, we can actually reduce this factored form to $G\overline{[f_4]} + H\overline{[f_2]}=\overline{[g]}$ where $\overline{[f]}$ is the coefficient vector of the alternating function $f$ with the zeros removed, that is 
\[
(G \ H) \begin{pmatrix}
\overline{[f_4]} \\ \overline{[f_2]}
\end{pmatrix}=\overline{[g]}.
\]

This system has a unique solution if and only if $(G \ H)$ is invertible for a generic choice of parameter values, i.e.~generically invertible.
\end{proof}

Note that for the matrix $(G \ H)$ to be generically invertible, it needs to be square and have full rank. Recall that the \textit{Sylvester matrix} associated to two polynomials $f(x)=a_nx^n+a_{n-1}x^{n-1} + \cdots +a_1x+a_0$ and $g(x) = b_mx^m + b_{m-1}x^{m-1} + \cdots + b_1x +b_0$ is the $n+m$ by $n+m$ matrix that has columns of the coefficients of $f(x)$ repeated $m$ times, and columns of the coefficients of $g(x)$ repeated $n$ times as:
\begin{center}
$\begin{pmatrix}
a_n & 0 & \cdots & 0 & b_m & 0 & \cdots & 0 \\
\vdots & a_n & \cdots & 0 & \vdots & b_m & \cdots & 0 \\
a_0 & \vdots & \cdots & \vdots & b_0 & \vdots & \cdots & \vdots \\
0 & a_0 & \cdots & 0 & 0 & b_0 & \cdots & 0 \\
\vdots & 0 & \cdots & a_n & \vdots & 0 & \cdots & b_m \\
\vdots & \vdots & \cdots & \vdots & \vdots & \vdots &\cdots & \vdots \\
0 & 0 & \cdots & a_0 & 0 & 0 & \cdots & b_0
\end{pmatrix}
$
\\
$\underbrace{\phantom{spacespacespace}}_m \underbrace{\phantom{spacespacespace}}_n$
\end{center}

The determinant of the Sylvester matrix of two polynomials is zero if and only if the two polynomials have a common root.  Thus for generic polynomials $f$ and $g$, 
the Sylvester matrix is invertible (see Chapter 3 of \cite{CLO2005} for relevant background on
resultants). 

Note that in the case of $(G \ H)$, this matrix is nearly the Sylvester matrix of two polynomials, though not exactly $f_1$ and $f_3$, but it possibly contains extra rows and columns. 

The following proposition and proof mirror that of Proposition 13 of \cite{mahdi_meshkat_sullivant_2014}.

\begin{prop}\label{prop:invert_if_square}
If the matrix $(G \ H)$ is square, then it is generically invertible.
\end{prop}

\begin{proof}
Suppose $(G \ H)$ is square.  We claim that the columns of $(G\ H)$ can be ordered in such a way that the block form of the matrix is
\[
\begin{pmatrix}
S' & 0 & 0 \\ 
X & S & Y \\
0 & 0 & S'' 
\end{pmatrix}
\]
where $S$ is the Sylvester matrix of $\hat{f_1}$ and $\hat{f_3}$ where $\hat{f}$ for an alternating polynomial $f$ is the polynomial with lowest degree zero and coefficient vector $\overline{[f]}$.  That is, $\hat{f}$ is a polynomial which does not alternate, with the same coefficients as $f$ (associated to different powers).  

We will show that either $(G \ H)$ is exactly the Sylvester matrix of generic polynomials $\hat{f_1}$ and $\hat{f_3}$, hence has full rank, or that one or both of $S'$ and $S''$ are 1 by 1 matrices with nonzero entry, and $S$ is the same Sylvester matrix, meaning that $(G \ H)$  continues to have full rank.
Note that the Sylvester matrix $S$ of these two polynomials $\hat{f_1}$ and $\hat{f_3}$ will have dimension $(n_1-m_1+n_3-m_3)/2$ by $(n_1-m_1+n_3-m_3)/2$.  
Recall from Equation \ref{eq:dim_GH} that $(G \ H)$ is a matrix of dimension 
\begin{equation}
\frac{\max\{n_1+n_4,n_2+n_3\} - \min\{m_1+m_4,m_2+m_3\}}{2}+1 \text{ by } \frac{n_2-m_2+n_4-m_4}{2}+2.
\end{equation}

Without loss of generality, we assume that $\max\{n_1+n_4,n_2+n_3\} = n_1+n_4$.  This occurs in one of three ways by Corollary \ref{cor:diff_eq_1}: 
\begin{itemize}
\item[(i)] $n_1=n_2+1$ and $n_4=n_3-1$, (in which case the two sums are equivalent)
\item[(ii)] $n_1=n_2-1$ and $n_4=n_3+1$, (in which case the two sums are equivalent) or
\item[(iii)] $n_1 = n_2+1$ and $n_4=n_3+1$.
\end{itemize}
In the first two cases, we do not add any zero rows above either $G''$ or $H''$ (described above) in making the matrices $G'$ and $H'$.  In either of these cases, we remove the row and column involving $S'$ from the block matrix.

In the last case, we add exactly two rows of zeros above the $H''$ matrix to make the matrix $H'$, and add no rows of zeros above $G''$ to make $G'$, hence we will have that $S'$ will be a 1 by 1 matrix with nonzero entry $a_{n_1}$, and $X$ is the remaining coefficients of $\hat{f_1}$ followed by zeros.

\smallskip

Now we consider the two possible cases of $\min\{m_1+m_4,m_2+m_3\} $.

First, suppose $\min\{m_1+m_4,m_2+m_3\} = m_1+m_4$.  This implies that the dimension of the $(G\ H)$ matrix is $(n_1+n_4-m_1-m_4)/2+1$ by $(n_2+n_4-m_2-m_4)/2+2$.

This can occur one of three ways by Corollary \ref{cor:diff_eq_1}:
\begin{itemize}
\item[(a)] $m_1=m_2+1$ and $m_4=m_3-1$, (in which case the two sums are equivalent)
\item[(b)] $m_1=m_2-1$ and $m_4=m_3+1$, (in which case the two sums are equivalent) or
\item[(c)] $m_1 = m_2-1$ and $m_4=m_3-1$.
\end{itemize}
In the first two cases, we do not add any zero rows below either $G''$ or $H''$ in making the matrices $G'$ and $H'$.  In either of these cases, we remove the row and column involving $S''$ from the block matrix above.

In the last case, we add exactly two rows of zeros below the $H''$ matrix to make the matrix $H'$, and add no rows of zeros below $G''$ to make $G'$.  Thus, we will have that $S''$ will be a 1 by 1 matrix with nonzero entry $a_{m_1}$, and $Y$ will continue with the other coefficients of $\hat{f_1}$ with zeros above.

Similarly, consider the case when $\min\{m_1+m_4,m_2+m_3\} = m_2+m_3$. Here, we have that the dimension of $(G \ H)$ is $(n_1+n_4-m_2-m_3)/2+1$ by $(n_2+n_4-m_2-m_4)/2+2$. This can occur one of three ways by Corollary \ref{cor:diff_eq_1}:
\begin{itemize}
\item[(A)] $m_1=m_2+1$ and $m_4=m_3-1$, (in which case the two sums are equivalent)
\item[(B)] $m_1=m_2-1$ and $m_4=m_3+1$, (in which case the two sums are equivalent) or
\item[(C)] $m_1 = m_2+1$ and $m_4=m_3+1$.
\end{itemize}

In the first two cases, we do not add any zero rows below either $G''$ or $H''$ in making the matrices $G'$ and $H'$.  In either of these cases, we remove the row and column involving $S''$ from the block matrix above.

In the last case, we add exactly two rows of zeros below the $G''$ matrix to make the matrix $G'$, and add no rows of zeros below $H''$ to make $H'$.  Thus, we will have that $S''$ will be a 1 by 1 matrix with nonzero entry $c_{m_3}$, and $Y$ will continue with the other coefficients of $\hat{f_3}$ with zeros above.

%
%

In any case, we have that $S$ is the Sylvester matrix of two polynomials with generic coefficients, namely $\hat{f_1}$ and $\hat{f_3}$, hence has full rank, and if $S'$ or $S''$ are in the block matrix, then they are 1 by 1 matrices with nonzero entries, hence have full rank.  Thus the matrix $(G \ H)$ has generic full rank, i.e.~is generically invertible.
\end{proof}

\begin{thm}\label{thm:LCsystem}
An LC system is locally identifiable if and only if the number of non-monic, nontrivial coefficients in the constitutive equation is equal to the number of parameters.
\end{thm}

\begin{proof}
Here, we show that if the number of parameters equals the number of non-monic, nontrivial coefficients, then the $(G \ H)$ matrix is square, hence by Propositions \ref{prop:goodquad_eq_ident}, \ref{prop:goodquad_if_invert}, and \ref{prop:invert_if_square} the model is locally identifiable.

Suppose $\mathcal{M}$ is an LC system which consists of a series combination of two smaller LC systems $N_1$ and $N_2$ with respective constitutive equations $f_1 V_1 = f_2 I_1$ and $f_3 V_2 = f_4 I_2$ where $f_1$ and $f_3$ are monic.  Also, suppose $f_i$ has shape $[m_i,n_i]$.  By induction, we suppose that the number of parameters equals the number of nontrivial, non-monic coefficients in both systems $N_1$ and $N_2$.  This implies that $N_1$ has $(n_1+n_2-m_1-m_2)/2+1$ parameters, and $N_2$ has $(n_3+n_4-m_3-m_4)/2+1$ parameters.

Assume the number of parameters equals the number of coefficients in the whole system, i.e.~
\begin{align*}
&\frac{n_1+n_2+n_3+n_4-m_1-m_2-m_3-m_4}{2}+2  \\&= \frac{\max\{n_1+n_4,n_2+n_3\} - \min\{m_1+m_4,m_2+m_3\} +n_1+n_3-m_1-m_3 }{2}+1
\end{align*}
Subtracting $(n_1+n_3-m_1-m_3)/2$ from both sides, we get
\begin{align*}
\frac{n_2+n_4-m_2-m_4}{2}+2  &= \frac{\max\{n_1+n_4,n_2+n_3\} - \min\{m_1+m_4,m_2+m_3\} }{2}+1
\end{align*}

This occurs exactly when the matrix $(G \ H)$ is square via Equation \ref{eq:dim_GH}.  
The argument for a parallel combination is identical, and omitted.
\end{proof}

\section{LCR System Analysis}\label{sec:LCR}

Now we consider the systems containing series and parallel combinations 
of all three base elements, that is LCR systems. 
We are not able to derive a complete classification of identifiability of these models,
and there already seem to be some significant challenges to generalizing the results
for two element systems to arbitrary LCR.  For example, there are general series-parallel LCR systems where there are more coefficients than the number of parameters.
Thus, there can be nontrivial relations between the coefficients in a general
series-parallel LCR system.  We will explore those equations in Section \ref{sec:equations}.
In this section, we look at basic properties of the general LCR systems,
including the numbers of types of systems in terms of the structure
of the constitutive equation.  We will show that there are $22$ different types.

To begin our study of general LCR systems,  we first 
consider several bounds on the orders of the constitutive equation.

\begin{prop}\label{prop:max_deg_lcr}
The maximum order of either side of the constitutive equation of an LCR system is bounded above by the number of parameters, i.e.~base elements, in the model.
\end{prop}

\begin{proof}
We will prove this statement inductively.  As the base case, note that the statement is true for each of our one element systems containing either a resistor, capacitor, or inductor.

Suppose that for LCR systems with less than $k$ base elements, the resulting constitutive equation has largest power less than or equal to the number of base elements.  

Now consider some LCR system $\mathcal{M}$ with $k$ base elements, 
which is a series combination of two smaller (in number of base elements) 
models which have $m$ and $n$ parameters respectively where $m+n=k$.  
By the inductive hypothesis, we know that the largest order of either side of 
the constitutive equations of the two smaller models are $m$ and $n$ respectively, 
i.e.~if $f_1V_1=f_2I_1$ and $f_3V_2=f_4I_2$ are the constitutive equations of the two models respectively, then $\deg(f_1),\deg(f_2) \leq n$ and $\deg(f_3),\deg(f_4) \leq m$.

Recall by Proposition \ref{prop:series_con} that the series combination of two systems with constitutive equations $f_1V_1=f_2I_1$ and $f_3V_2=f_4I_2$ yields constitutive equation:
\[
f_1f_3V=(f_1f_4+f_2f_3)I.
\]

Therefore, the largest power of either side of the constitutive equation is $n+m=k$.  
By duality, we also have the result when $\mathcal{M}$ is a parallel combination.
\end{proof}

\begin{prop} \label{prop:LCR_diff_one}
In an LCR system, the largest orders on either side of the constitutive equation must be within one of each other.  Similarly, the smallest orders on either side of the constitutive equation must be within one of each other.
\end{prop}

\begin{proof}
We prove the proposition using induction on the number of base elements in the system.  As the base case, note that in both the inductor, and the capacitor base element, the difference in the largest power between the two sides of the constitutive equation is one.  Similarly, since the smallest power is the largest power, that difference is also one.  In the case of a single resistor, both sides have a single element of order zero.

Suppose then that the statement of the proposition is true for LCR systems with less than $k$ base elements.
Then suppose that $M$ is an LCR system with $k$ elements which is generated by, without loss of generality, a series combination of two strictly smaller (in terms of number of base elements) systems $N_1$ and $N_2$.  Suppose that $N_1$ and $N_2$ have respective constitutive equations $f_1V_1=f_2I_1$ and $f_3V_2=f_4I_2$ where each $f_i$ is defined just as in Equation \ref{eq:fourfs}.

By the inductive hypothesis, since $N_1$ and$N_2$ have less than $k$ base elements, then we know that $|n_1-n_2| \leq 1$, $|m_1-m_2| \leq 1$, $|n_3-n_4|\leq 1$, and $|m_3-m_4| \leq 1$.  Note that by Equation \ref{eq:con_series}, the constitutive equation of the system $M$ generated by combining $N_1$ and $N_2$ in series is
\begin{equation}\label{eq:diff1_series}
f_1f_3V = (f_1f_4+f_2f_3)I.
\end{equation}

Thus, we have that the maximal order of the left-hand side of the Equation \ref{eq:diff1_series} is $n_1+n_3$, while the maximal order of the right-hand side is $\max\{n_1+n_4,n_2+n_3\}$.  Therefore, the difference in the largest order of either side of the constitutive equation of $M$ is 
\begin{align*}
|n_1+n_3-\max\{n_1+n_4,n_2+n_3\}|&=|\min\{n_1+n_3-n_1-n_4,\ n_1+n_3-n_2-n_3\}| \\
&= |\min \{n_3-n_4,\ n_1-n_2\}|
\end{align*}
Note that in either case, we have that $|n_1-n_2|\leq 1$ and $|n_3-n_4|\leq 1$, thus the difference of the maximal order of either side of the constitutive equation of $M$ is at most one.

The bound for the minimal order is similar (with maximums and minimums swapped) and is omitted.
The bounds also follow for parallel combinations by circuit duality.



Thus, by induction, LCR systems have the difference of the largest order of either side of their constitutive equations at most one, and the difference of the smallest order of either side of their constitutive equations at most one.
\end{proof}

\begin{rem}
The main difference between all of the two element systems and the three element system is that we no longer have a bound on the number of coefficients of the constitutive equation by the number of parameters.  In fact, by the previous two propositions, we can have up to $2n+1$ nonzero, non-monic coefficients in the constitutive equation of an LCR system with $n$ base elements.  As a result of this lack of a bound, we could have systems with more coefficients than base elements which are locally identifiable.  This is not entirely surprising, however the existence of systems with more base elements than coefficients which are \textbf{not} locally identifiable leads us to believe that to find a sufficient condition for the identifiability of an LCR system, we need to look beyond comparing the number of coefficients to the number of parameters.
\end{rem}

\begin{eg} 
Consider the LCR system depicted in Figure \ref{fig:7}.
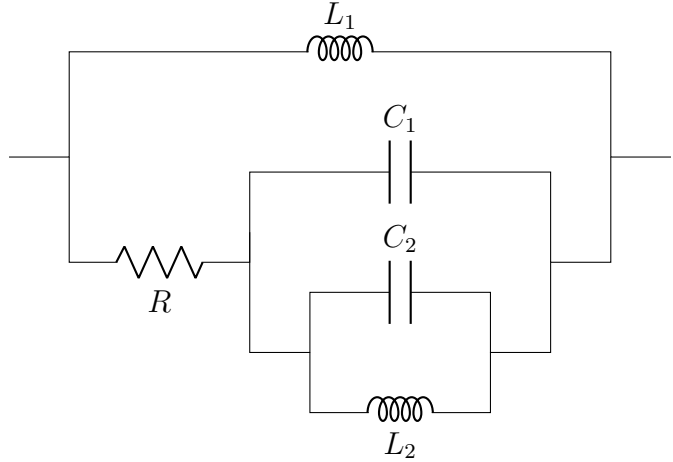
\begin{figure}[H] 
\begin{circuitikz}[scale=.8]
\draw 
(0,.25) to (1,.25) to (1,2) to[L,l=$L_1$] (10,2) to (10,.25) to (11,.25);   
\draw (1,.25) to (1,-1.5) to[R,a=$R$] (4,-1.5) to (4,0) to[C,l=$C_1$] (9,0) to (9,-1.5) to (10,-1.5) to (10,.25);
\draw (4,-1) to (4,-3) to (5,-3) to (5,-2) to[C,l=$C_2$] (8,-2) to (8,-3) to (9,-3) to (9,-1.5);
\draw (5,-3) to (5,-4) to[L,a=$L_2$] (8,-4) to (8,-3);
\end{circuitikz}
\caption{An LCR system $L_1 \vee(R \wedge (C_1 \vee (L_2 \vee C_2))))$.}
\label{fig:7}
\end{figure}
This system has constitutive equation:

\begin{multline*}
(C_1L_1L_2+C_2L_1L_2)V^{(3)} + (C_1L_2R+C_2L_2R)\ddot{V}+(C_1C_2L_1+C_1C_2L_2)\dot{V}+ C_1C_2RV \\
= (C_1L_1L_2R+C_2L_1L_2R)I^{(3)} + C_1C_2L_1L_2 \ddot{I} + C_1C_2L_1R\dot{I}.
\end{multline*}

Note that this LCR system has five parameters and after normalization has six non-monic, nonzero coefficients in its constitutive equation.  If we consider the Jacobian matrix of the map from the space of parameters to the space of coefficients of the constitutive equation corresponding to this example, we see that the rank of the Jacobian is non-maximal, meaning the system is not identifiable.  This example first shows that the number of coefficients in a constitutive equation of an LCR system is not bounded by the number of parameters, and moreover having at least as many coefficients as parameters in an LCR system is not a sufficient condition for local identifiability.
\end{eg}

We now introduce a similar notion of types as in the LC systems to more general LCR systems.

\begin{defn}\label{def:type}
Let $\mathcal{M}$ be an LCR system with constitutive equation $f_1V=f_2I$ where 
$f_1 = a_{n_1}x^{n_1}+ \cdots + a_{m_1}x^{m_1}$ and $f_2 = b_{n_2} x^{n_2} + \cdots +b_{m_2}x^{m_2}$.  Then we define the \textit{type} of $\mathcal{M}$ as the quadruple 
$(m_1-m_2,\ n_1-n_2, \ c, \ d)$ where $c,d=1$ if
$f_1$ and $f_2$ have the alternating property respectively, and are 0 otherwise.
\end{defn}

\begin{eg}\label{eg:base_types}
The three base elements can be characterized by type, where because there is only a single nonzero coefficient on either side of the constitutive equation, each side of all three constitutive equations are defined as alternating.  More explicitly, by Equations \ref{eq:res}, \ref{eq:ind}, and \ref{eq:cap}, we have that the resistor, inductor, and capacitor have respective types $(0,0,1,1)$, $(-1,-1,1,1)$ and $(1,1,1,1)$.
\end{eg}

\begin{rem}
The four types, $A,B,C,D$ of LC systems described in Section \ref{sec:LC} can be generalized as $A=(-1,-1,1,1)$, $B=(-1,1,1,1)$, $C=(1,-1,1,1)$ and $D=(1,1,1,1)$ as LCR types.
\end{rem}

Note that there are certain restrictions on what this type quadruple can look like. 
For example we know that because of Proposition \ref{prop:LCR_diff_one}, 
the first two entries of the type must both be in the set $\{-1,0,1\}$.  
We have not yet shown that for an LCR system that the left and right-hand 
sides of the constitutive equation must strictly have the alternating property 
or the saturated property.  More precisely, it is not obvious that a differential operator in the constitutive equation of an LCR system cannot skip an order without having the alternating property, i.e.~have an order with zero coefficient, but not have the remaining even or odd orders also have zero coefficients.  We now prove that this is in fact the case, and conclude that all LCR systems fall into one of these types.  To do this, we first need the following Lemma.

\begin{lem}\label{lem:firstbadtype}
No LCR system can have constitutive equation of type $(*,-1,0,1)$ for any entry of $*$.
\end{lem}

\begin{proof}
Suppose $\mathcal{M}$ is an $n$ base element LCR system with type of the form $(*,-1,0,1)$, and constitutive equation $f_1V=f_2I$.  Note that $f_2$ must alternate since the fourth entry of the type quadruple is 1, and also $f_2$ must have largest order one larger than $f_1$ since the second entry of the type quadruple is $-1$.  Similarly, $f_1$ must not alternate since the third entry of the quadruple is $0$.

Recall that the three base elements, the resistor, inductor, and capacitor, have respective types $(0,0,1,1)$, $(-1,-1,1,1)$ and $(1,1,1,1)$, hence $\mathcal{M}$ cannot be a base element, i.e.~$n\geq 1$.  Therefore, $\mathcal{M}$ must be made of some series or parallel combination of two systems with strictly fewer elements, say $A_1$ and $A_2$.  Let $A_1$ and $A_2$ have constitutive equations $g_1V_1=g_2I_1$ and $g_3V_2=g_4I_2$ respectively.  

First, if we suppose $\mathcal{M}$ is a series combination of $A_1$ and $A_2$, then we have that $f_1=g_1g_3$ and $f_2=g_1g_4+g_2g_3$.  Note that since $f_2$ must alternate, then all four of $g_1,g_2,g_3$ and $g_4$ must alternate (and have some parity conditions), however $f_1$ must not alternate meaning that one of $g_1$ or $g_3$ cannot alternate, a contradiction.  Thus, an LCR system of type $(*,-1,0,1)$ cannot be constructed via a series combination of other systems with fewer elements.

Now suppose $\mathcal{M}$ is a parallel combination of $A_1$ and $A_2$.  In this case, we have that $f_1=g_1g_4+g_2g_3$ and $f_2=g_2g_4$.  Since $f_2$ must have one higher largest order than $f_1$, we must have that $\deg(g_1) < \deg(g_2)$ and $\deg(g_3) < \deg(g_4)$.  Thus, by Proposition \ref{prop:LCR_diff_one}, we have that $\deg(g_2)=\deg(g_1)+1$ and $\deg(g_4)=\deg(g_3)+1$, meaning that $A_1$ and $A_2$ have respective types with second entry both being $-1$.

Given that $f_2$ is alternating, we must have that both $g_2$ and $g_4$ are alternating, i.e.~both $A_1$ and $A_2$ have a 1 in the last entry of their types.  Similarly, given that $f_1$ is not alternating, we must have either $g_1$ or $g_3$ not alternating, or that $g_1g_4$ and $g_2g_3$ have opposite parity (the sum zips together).  Note though that $g_1g_4$ and $g_2g_3$ must have the same parity, since the pairs $g_1,g_2$ and $g_3,g_4$ must have different parity (because their maximal orders have a difference of exactly one from above), meaning $g_1g_4$ has even (odd) parity if and only if $g_2g_3$ has even (odd) parity.  Thus, $A_1$ and $A_2$ must have types of the form $(*,-1,r_1,1)$ and $(*,-1,r_2,1)$ where at least one of $r_1$ or $r_2$ is equal to $0$.

Therefore, the only way to generate a system of type $(*,-1,0,1)$ is by a parallel combination of two systems, one of which has type $(*,-1,0,1)$, but since none of the base elements have this type, then this type must not exist. 

\end{proof}

\begin{prop}[Skipping but not alternating]\label{prop:skip_bn_alt}
Let $M$ be a series-parallel LCR system.  Then each side of the 
constitutive equation must either be alternating or saturated.
\end{prop}


\begin{proof}
We prove this statement by induction on the number of base elements in the system.  As the base case, note that in each of the one-element systems, the statement is certainly true, as each side only has a single nonzero coefficient.

Now, suppose the statement is true for LCR systems with less than $k$ base elements, that is, either side of the constitutive equation for LCR systems with less than $k$ base elements cannot skip an order without having that side alternate.   Also, suppose $\mathcal{M}$ is an LCR system with $k$ base elements, which is generated by a series combination of two smaller systems $N_1$ and $N_2$ with strictly less than $k$ base elements.  Let $N_1$ and $N_2$ have constitutive equations $f_1V_1 = f_2I_1$ and $f_3V_2=f_4I_2$ respectively, and note by the inductive hypothesis, $f_1,f_2,f_3$ and $f_4$ cannot skip a coefficient without alternating.  Therefore, each of $N_1$ and $N_2$ can be characterized by a type as described in Definition \ref{def:type}.
Let us define each polynomial $f_1,f_2,f_3$ and $f_4$ as in Equation \ref{eq:fourfs}.

By Proposition \ref{prop:LCR_diff_one}, we know that the first two entries of each $N_i$ system's type is either $-1,0$ or $1$.

Also, we have that the constitutive equation of $\mathcal{M}$ is $f_1f_3 V = (f_1f_4+f_2f_3)I$.  Note that for each of the three products of two $f_i$, if both polynomials in the product alternate, then the resulting product alternates.  Also, if at least one of the polynomials in the product is saturated, then by Lemma \ref{lem:prod_nonalt}, we know that the resulting product is saturated.  Thus, we immediately have that the left-hand side of the constitutive equation, i.e.~the product $f_1f_3$, cannot skip without alternating.  We also know that each element of the sum of the right-hand side cannot skip without alternating, hence to finish the proof we need only show that their sum cannot skip without alternating.  

Note that by Proposition \ref{prop:LCR_diff_one}, we know that the absolute difference in the largest orders, and the absolute difference of the smallest orders of $f_1f_4$ and $f_2f_3$ are both at most two.  

Thus, the only way that the right-hand side of the constitutive equation for $\mathcal{M}$ could skip an order without alternating is if one of the elements of the sum had maximal (or minimal) order two larger (smaller) than the other, and the one with larger maximal order alternates while the other is saturated.  This would result in skipping the second largest (smallest) order of the sum, but the rest of the sum having nonzero coefficients.  

Suppose without loss of generality that $f_1f_4$ has largest order two larger than $f_2f_3$ and suppose $f_1$ and $f_4$ alternate, while at least one of $f_2$ and $f_3$ is saturated. Thus, the third entry of the type of $N_1$ and the fourth entry of the type of $N_2$ must be 1.  Also, either the fourth entry of the type of $N_1$ or the third entry of the type of $N_2$ must be zero.  Without loss of generality we suppose that it is $f_3$ that is saturated.  Note that for $f_1f_4$ to have largest order two larger than $f_2f_3$, we must have that $n_1=n_2+1$ and $n_4=n_3+1$, i.e.~$f_1$ and $f_4$ have largest order one larger than $f_2$ and $f_3$ respectively by Proposition \ref{prop:LCR_diff_one}.   Thus, the second entry in the type of $N_1$ must be a $1$, while the second entry in the type of $N_2$ must be a $-1$.  Therefore, $N_1$ must have type $(*,1,1,*)$ and $N_2$ must have type $(*,-1,0,1)$ where the $*$ represents any possible entry.

Since the only way to have a constitutive equation which skips but does not alternate is to have an element of type $(*,-1,0,1)$ which does not exist by Lemma \ref{lem:firstbadtype}, then there cannot be a constitutive equation which skips but does not alternate with $k$ base elements.  Thus, by induction, no LCR system can skip but not alternate.

The case of a parallel combination follows from circuit duality.
\end{proof}

\begin{lem}\label{lem:prod_nonalt}
If $f$ and $g$ are polynomials with non-negative coefficients such that $f$ is saturated and $g$ is alternating, then the product $fg$ is saturated.
\end{lem}

\begin{proof}
Suppose $f$ is saturated, and $g$ is alternating, such that they have form
\begin{align*}
f&=a_nx^n+a_{n-1}x^{n-1}+\cdots +a_{m+1}x^{m+1} + a_{m}x^m \\
g&=b_vx^v+b_{v-2}x^{v-2}+\cdots +b_{u+2}x^{u+2} + b_{u}x^u 
\end{align*}

Note that the product of $f$ and $g$ has form
\[
fg = \sum_{k=m+u}^{n+v} \left( \sum_{i+j=k} a_ib_j \right) x^k.
\]

Thus, to show the statement of the Lemma is true, we need only show that for each $k$, there is some nonzero combination of coefficients from $g$ and $f$ with corresponding degree adding to $k$.  Given that $g$ alternates and $f$ is saturated, and all coefficients are non-negative, this problem equates to the following:
Given the sets of non-negative integers $\mathcal{F}=\{n,n-1,\ldots, m+1,m\}$ and
 $\mathcal{G}=\{v,v-2,\ldots , u+2,u\}$, for every integer $k$ with $m+u \leq k \leq n+v$,
can we find a sum of an element $i\in \mathcal{F}$ and an element $j \in \mathcal{G}$ such that $i+j=k$?  The answer to this question is yes, as we can generate every number from $m+u$ to $n+v$ as 
\[
m+u, (m+1)+u, m+(u+2), (m+1)+(u+2), \ldots, (n-1)+(v-2), n+(v-2), (n-1)+v,n+v.
\]

Thus, for each $k$ there is some $a_i$ and $b_j$ such that $a_ib_j\neq 0$ and $i+j=k$, hence $fg$ is saturated, as desired.
\end{proof}

\begin{cor}\label{cor:alltyped}
Every LCR system has one of the types as defined in Definition \ref{def:type}.
\end{cor}

\begin{proof}
The only way that we could not classify an 
LCR system with a type would be if it had a constitutive equation which 
did not alternate, but also was not saturated by our definition, i.e.~that skipped without alternating.  By Proposition \ref{prop:skip_bn_alt}, this cannot happen.  Thus, every LCR system can be characterized by a type.
\end{proof}

We can now make several more statements about the type characterization we propose for LCR systems.

\begin{prop}\label{prop:type_series}
We can characterize the type of a series combination of two LCR systems of types $(a,b,c,d)$ and $(e,f,g,h)$ respectively as 
\begin{equation}\label{eq:series_type}
(a,b,c,d) \odot (e,f,g,h) = \left( \max\{a,e\}, \ \min\{b,f\},\ cg, \ cdgh(1-||a|-|e||) \right).
\end{equation}
\end{prop}

\begin{proof}
Suppose $\mathcal{M}$ is generated by a series combination of two smaller LCR systems $N_1$ and $N_2$ with respective constitutive equations $f_1 V_1 = f_2 I_1$ and $f_3 V_2 = f_4 I_2$.  Note then that the constitutive equation of $\mathcal{M}$ is $f_1f_3 V = (f_1 f_4 + f_2 f_3) I$.  

Let us define each polynomial $f_1,f_2,f_3,$ and $f_4$ just as in Equation \ref{eq:fourfs}.

Note that the first entry in the type of $\mathcal{M}$ is the difference in the smallest orders of both sides of its constitutive equation, i.e.~
\begin{align*}
(m_1+m_3) - \min\{m_1+m_4,\ m_2+m_3\} &= \max\{m_1+m_3-m_1-m_4, \ m_1+m_3-m_2-m_3\} \\
&=\max\{m_3-m_4,\ m_1-m_2\} \\
&= \max\{e,a\}.
\end{align*}

Similarly, the second entry in the type of $\mathcal{M}$ is the difference in the largest orders of both sides of its constitutive equation, i.e.~
\begin{align*}
(n_1+n_3) - \max\{n_1+n_4,\ n_2+n_3\} &= \min\{n_1+n_3-n_1-n_4, \ n_1+n_3-n_2-n_3\} \\
&=\min\{n_3-n_4,\ n_1-n_2\} \\
&= \min\{f,b\}.
\end{align*}

Also, $f_1f_3$ alternates if and only if $f_1$ and $f_3$ both alternate, i.e.~the third entry of the type of $\mathcal{M}$ is 1 if and only if both third entries of the types of $N_1$ and $N_2$ are 1.  This is true exactly when $c=g=1$, equivalently if and only if $cg=1$.

Finally, the right hand side of the constitutive equation of $\mathcal{M}$ alternates if and only if all four of the $f_i$ alternate, and $f_1f_4$ has the same parity as $f_2f_3$.  Note that these two products have the same parity if and only if either all $f_i$ have the same parity, or if $f_1$ and $f_2$ have different parity and $f_3$ and $f_4$ have different parity.  More explicitly, we can consider the smallest order of all of the alternating $f_i$ and note that these two products have the same parity if and only if either each of $a=e=0$, or $|a|=|e|=1$.
Thus, the fourth entry of the type of $\mathcal{M}$ is 1 if and only if all of the third and fourth entries of $N_1$ and $N_2$ are 1 and $||a|-|e||=0$, i.e.~$1-||a|-|e||=1$.
\end{proof}

By circuit duality, we also get a similar formula for parallel combinations.

\begin{prop}\label{prop:types_para}
We can characterize the type of a parallel combination of two LCR systems of types $(a,b,c,d)$ and $(e,f,g,h)$ respectively as 
\begin{equation}\label{eq:para_type}
(a,b,c,d) \oplus (e,f,g,h) = \left( \min\{a,e\}, \ \max\{b,f\},\  cdgh(1-||a|-|e||), \ dh \right).
\end{equation}
\end{prop}

Given our type characterization and the restrictions imposed on the type by Propositions \ref{prop:LCR_diff_one} and \ref{prop:skip_bn_alt}, there are 36 possible types of the form $\left( \{-1,0,1\}, \ \{-1,0,1\}, \ \{0,1\},\ \{0,1\} \right)$.  Note however that not all 36 of these quadruples correspond to types of LCR systems which are generated by series and parallel combinations of the three base elements.  We can generate all possible types by implementing a recursive algorithm starting with a generating set consisting of the three base element types $(0,0,1,1)$, $(-1,-1,1,1)$, and $(1,1,1,1)$, generating every possible combination of these types, and adding these combinations to the generating set.  Repeating this process until no new quadruples are added to the generating set, we then have all possible types.

 
\begin{prop} The following $22$ quadruples are the only possible LCR types:
\begin{center}
\begin{tabular}{cccccc}
$(1, 0, 0, 0)$, 
&$(-1, 0, 0, 0)$,
&$(0, 0, 1, 1)$,
&$(1, -1, 1, 0)$,
&$(0, 1, 0, 1)$,
&$(0, -1, 0, 0)$,
\\$(1, 1, 0, 0)$,
&$(-1, 1, 1, 1)$,
&$(1, 1, 1, 1)$,
&$(0, 0, 0, 1)$,
&$(0, 0, 1, 0)$,
&$(-1, -1, 0, 0)$,
\\$(1, -1, 0, 0)$,
&$(0, 1, 0, 0)$,
&$(-1, 1, 0, 1)$,
&$(0, -1, 1, 0)$,
&$(0, 0, 0, 0)$,
&$(-1, -1, 1, 1)$,
\\&$(1, -1, 1, 1)$,
&$(1, 0, 1, 0)$,
&$(-1, 0, 0, 1)$,
&$(-1, 1, 0, 0)$.
\end{tabular}
\end{center}

\end{prop}

To conclude this section, we give an example that shows that type analysis, as
was performed to analyze the two component systems, is not sufficient to characterize the
identifiability of general series-parallel LCR systems.

\begin{eg}
Consider the model $M = (R_1 \vee C) \wedge (R_2 \vee L)$.
This model has constitutive equation
\[
R_1 L \ddot{V}  + (CL + R_1R_2) \dot{V} + CR_2 V  =  L R_1 R_2 \ddot{I} + (LCR_1 + LCR_2) \dot{I} + CR_1R_2 I. 
\]
This model is saturated on both sides, and the shapes of the differential operators are
$[0,2]$ and $[0,2]$.  Since the constitutive equation
has the same highest and lowest order on both sides,
this model has type $(0,0,0,0)$.  

Now consider the model $N = M \wedge R_3$ obtained by joining
$M$ in series to a new resistor $R_3$.  The model of a single resistor has type $(0,0,1,1)$,
so the model $N$ will also have type $(0,0,0,0) \odot (0,0,1,1) = (0,0,0,0)$,
and the differential operators also have shapes $[0,2]$ and $[0,2]$.
In this case there are five parameters and five non-monic coefficients, and
a direct calculation shows that the model is locally identifiable.

Finally, consider the new model $N' = N \wedge R_4$ obtained by joining
$N$ in series to a new resistor $R_4$.   Again this model $N'$ will have type
$(0,0,0,0) \odot (0,0,1,1) = (0,0,0,0)$ and the differential operators 
also have shapes $[0,2]$ and $[0,2]$.  But now the model cannot be identifiable
because there are six parameters and there continue to be only five non-monic coefficients.
This shows that the combinations of types $(0,0,0,0) \odot (0,0,1,1)$
may or may not be identifiable depending on the structure of the underlying model.
\end{eg}


\section{Equations Defining LCR Models}\label{sec:equations}

General LCR models can have more non-monic coefficients than
the number of parameters.  Hence, the set of constitutive equations consistent with
a particular model $M$ will be a subset of all possible differential equations
of a given type.  Understanding the algebra and geometry of these
sets of constitutive equations is an interesting problem, and might be useful
for addressing identifiability questions for general LCR circuit systems.

\begin{eg}\label{ex:det}
Consider the LCR system $M = (R \vee C) \wedge L$.
The constitutive equation in this case is
\[
R \dot{V} + V  =   RL  \ddot{I} + L \dot{I} + RC I.
\]
Note that there are three parameters and four non-monic coefficients.
Hence, not every constitutive equation of shape
\[
c_1 \dot{V} + c_0 V  =    d_2  \ddot{I} + d_1 \dot{I} + d_0 I
\]
with positive coefficients can arise from some choice of $R, C, L$.
To describe the relations that arise, we find it useful to work in the projective
representation, as this will produce homogeneous equations.  In this case,
the projective version of the constitutive equation is
\[
L_0C_0R_1 \dot{V} + L_0C_0R_0 V  =   L_1C_0R_1  \ddot{I} + L_1C_0R_0 \dot{I} + L_0C_1R_1 I.
\]
Note that these coefficients satisfy the relation:  $c_1 d_1 = c_0 d_2$.  
\end{eg}

\begin{eg}\label{ex:res}
Consider the four element model $M = (R_1 \wedge C) \vee (R_2 \wedge L)$.  The
constitutive equation is
\[
R_1 L \ddot{V} + (CL + R_1 R_2) \dot{V} + CR_2 V = 
R_1R_2L \ddot{I} + (C R_2 L  + C R_1 L) \dot{I} + R_1 R_2 C I.
\]
There are six coefficients and four parameters.  In the projective version, we expect a single 
homogeneous equation that defines the relations on the coefficients.
It is
\[
c_0^2 d_2^2 - c_1c_0d_2d_1 + c_2c_0d_1^2 + \underline{2c_2c_0d_2d_0} - c_2c_1d_1d_0 + c_2^2 d_0^2 =0.
\]
This polynomial is remarkably similar looking to the resultant of
the two quadratic polynomials $c_2 x^2 + c_1 x + c_0$ and
$d_2 x^2 + d_1 x + d_0$.  However, the sign of the underlined term is wrong.
It is unclear if this polynomial can be expressed as the resultant of related polynomials.
We also do not know if every $6$-tuple $(c_2, c_1, c_0, d_2, d_1, d_0) $ of positive
numbers that satisfies this equation can come from some choice of positive values for
$C, L, R_1, $ and $R_2$.
\end{eg}

Examples \ref{ex:det} and \ref{ex:res} just give a small taste of the types of
equations that can arise.  We do not have a general theory of what those equations
should look like, but we can try to derive properties of the ideals in the hopes
of understanding their structure.

In general, associated to any series-parallel model $M$ is a homogeneous ideal
\[
I_M \subseteq \R[{\bf c}, {\bf d}] = \R[c_0, c_1, \ldots, c_m, d_0, d_1, \ldots, d_m].
\]
For example, in Example \ref{ex:det}, we get that $I_M = \langle c_1 d_1 - c_0 d_2 \rangle$.
In fact, beyond being just an ordinary homogeneous ideal, $I_M$ satisfies some
other homogeneities as well.

Call a polynomial $p(c,d) \in \R[{\bf c},{\bf d}]$ bihomogeneous, if it is homogeneous in each
set of variables, that is $p(\lambda c, \delta d)  =  \lambda^m \delta^n p(c,d)$ for some
$m$ and $n$.  The pair $(m,n)$ is called the bidegree of $p$.
An ideal $I \in \R[{\bf c},{\bf d}]$ is bihomogeneous if it has a generating set consisting
of bihomogeneous polynomials.  The notion of bihomogeneity of ideals also can be interpreted
naturally in terms of the corresponding variety, at least when $I$ is radical.
Let $V = V(I) \subseteq  \R^{2n + 2}$ be the corresponding variety of
pairs $(\bfc, \bfd)$ coming from the model.  
Bihomogeneity of the radical ideal $I \subseteq \R[\bfc, \bfd]$  
is equivalent  to the following condition on the variety $V =  V(I)$:
for any pair $(\bfc, \bfd) \in V$ and any nonzero $\lambda, \delta \in \C$,  
$(\lambda \bfc, \delta \bfd)$ is also
in $V$.

\begin{prop}\label{prop:bihomogeneous}
For any series-parallel circuit network $M$, the vanishing ideal $I_M$ is bihomogeneous
in ${\bf c}$ and ${\bf d}$.  
\end{prop}

\begin{proof}
We proceed by induction on the number of components in the network.  
The statement is clearly true if the networks have just one component, since the vanishing
ideal is the zero ideal in that case.

By symmetry, we can suppose that the model is a series combination $M = M_1 \wedge M_2$.
By induction, we can suppose that $M_1$ and $M_2$ satisfy the bihomogeneity assumption.
For two sequences $\bfc = (c_0, c_1,c_2, \ldots )$ and $\bfd = (d_0, d_1, d_2, \ldots)$
let $\bfc \ast \bfd$ denote the convolution
\[
\bfc \ast \bfd =  (c_0d_0, c_1d_0 + c_0 d_1, c_2d_0 + c_1d_1 + c_0d_2, \ldots). 
\]
Then, with this operation defined, we have that
\[
M  = \{ (\bfc \ast \bfc',  \bfc \ast \bfd' + \bfc' \ast \bfd) :  (\bfc, \bfd) \in M_1, (\bfc', \bfd') \in M_2 \}. 
\]
So, we need to show that if $(\bfc \ast \bfc',  \bfc \ast \bfd' + \bfc' \ast \bfd) \in M$ 
and if $\lambda, \delta \in \C^\ast$ then $(\lambda(\bfc \ast \bfc'),  
\delta (\bfc \ast \bfd' + \bfc' \ast \bfd) ) \in M$.  But by the inductive hypothesis, we know that if $(\bfc, \bfd) \in M_1$, 
$(\bfc', \bfd') \in M_2$, and $\lambda, \delta, \lambda', \delta'$ are nonzero then
\[
( \lambda \lambda' \bfc \ast \bfc',  \lambda \delta' \bfc \ast \bfd' + \lambda' \delta \bfc' \ast \bfd)
\in M.
\]
Taking $\lambda = \lambda$, $\lambda' = 1$, $\delta = \delta$ and $\delta' = \delta/\lambda$
gives the desired result.
\end{proof}

A second type of homogeneity also holds for the vanishing ideals of 
circuit models.  We introduce a grading on the polynomial ring $\R[{\bf c}, {\bf d}]$
by setting $\deg(c_i) = \deg(d_i) = i$.
The degree of a monomial $\deg(c^\alpha d^\beta)$ is the sum of the degrees
of all variables in the monomial, counted with multiplicity.  So, for example,
\[
\deg(c_1^2 c_3 d_0d_4) = 1 + 1 + 3 + 0 + 4 = 9.
\]  
A polynomial in $\R[{\bf c}, {\bf d}]$ in the degree grading is called
homogeneous if every monomial appearing has the same degree.  An ideal
$I_M$ is degree homogeneous if it has a generating set consisting of degree homogeneous
polynomials.
The notion of degree homogeneity can also be interpreted in terms 
of the corresponding variety, at least when the ideal is radical.
Degree homogeneity of the radical ideal $I \subseteq \R[\bfc, \bfd]$  
is equivalent  to the following condition on the variety $V =  V(I)$:
for any pair $(\bfc, \bfd) \in V$ and any nonzero $\lambda \in \C$,  
$(\lambda^0 c_0, \lambda^1 c_1, \lambda^2 c_2, \ldots,  
\lambda^0 d_0, \lambda^1 d_1, \lambda^2 d_2,)$ is also
in $V$.  Denote the operation of applying $\lambda$ to $(\bfc, \bfd)$ in this
way by $\lambda \cdot (\bfc, \bfd) =  (\lambda \cdot \bfc, \lambda \cdot \bfd)$.  

\begin{prop}
For any series-parallel circuit network $M$, the vanishing ideal $I_M$ is degree homogeneous
in ${\bf c}$ and ${\bf d}$.  
\end{prop}

\begin{proof}
We proceed by induction on the number of components in the network.  
The statement is clearly true if the networks have just one component, since the vanishing
ideal is the zero ideal in that case.

By symmetry, we can suppose that the model is a series combination $M = M_1 \wedge M_2$.
By induction, we can suppose that $M_1$ and $M_2$ satisfy the degree homogeneity assumption.
As in the proof of Proposition \ref{prop:bihomogeneous}, we need to show that
if $(\bfc \ast \bfc',  \bfc \ast \bfd' + \bfc'  \ast \bfd) \in M$
and $\lambda \in \C^*$ then $(\lambda  \cdot (\bfc \ast \bfc'),  \lambda \cdot (\bfc  \ast \bfd' + \bfc' \ast \bfd) \in M$.  Note that $\cdot$ and $\ast$ interact in the following way: 
\[
(\lambda \cdot \bfc) \ast (\lambda \cdot \bfc') =  \lambda \cdot (\bfc \ast \bfc')
\]
with similar expressions holding for other combinations of $\bfc, \bfd, \bfc', \bfd'$.  
Since
$\lambda \cdot (\bfc, \bfd) = 
(\lambda \cdot \bfc, \lambda \cdot \bfd) \in M_1$ and $\lambda \cdot (\bfc', \bfd') =
(\lambda \cdot \bfc', \lambda \cdot \bfd') \in M_2$
we get that
\[
((\lambda \cdot \bfc ) \ast   (\lambda \cdot \bfc'), (\lambda \cdot \bfc ) \ast   (\lambda \cdot \bfd') + (\lambda \cdot \bfc' ) \ast   (\lambda \cdot \bfd)  ) =  
 (\lambda  \cdot (\bfc \ast \bfc'),  \lambda \cdot (\bfc  \ast \bfd' + \bfc' \ast \bfd) \in M
\]
which is the desired result.
\end{proof}


\section{Future Work and Discussion}\label{sec:future}
In general, determining local identifiability of LCR systems appears to require novel techniques as compared to those of the two base element subsystems.  Due to the lack of an upper bound on the number of coefficients of the constitutive equation by the number of parameters, it seems that the problem of determining identifiability of LCR systems must extend past the counting of parameters and coefficients.  

Another approach to studying identifiability via a single constitutive equation is to consider the generation of the coefficients themselves, and consider the Jacobian of the map from the parameter space to these coefficients \cite{MeshkatSullivant}.  In the case of the LCR system, this initially seems difficult due to the lack of an obvious pattern in the coefficients based on the graph structure, though more study could be done on this subject.  Perhaps the electromechanical analogy of the LCR system into a mechanical system could lead to a more natural study of the coefficients of the related constitutive equation.


\section*{Acknowledgments}

Cashous Bortner and Seth Sullivant were partially supported
by the NSF (DMS 1615660).

\bibliographystyle{plain}
\bibliography{LCR}

\end{document}